\documentclass[10pt]{amsart}

\usepackage{amsmath}
\usepackage{amsfonts}
\usepackage{amssymb}
\usepackage{graphicx}
\usepackage{color}

\newcommand{\Bignorm}[1]{\Bigl\|#1\Bigr\|}

\newcommand{\NN}{{\mathbb N}}
\newcommand{\bC}{{\mathbb C}}
\newcommand{\Rst}{{\mathbb R}}

\newcommand{\Rdst}{{\Rst^d}}

\newcommand{\Rtdst}{{\Rst^{2d}}}
\newcommand{\set}[2]{\big\{ \, #1 \, \big| \, #2 \, \big\}}
\newcommand{\norm}[1]{\lVert#1\rVert}
\newcommand{\Esp}{E}

\newcommand{\Zst}{{\mathbb Z}}
\newcommand{\Zdst}{{\Zst^d}}
\newcommand{\supess}{\mathop{\operatorname{esssup}}}

\newcommand{\Lsp}{L}
\newcommand{\Ltsp}{{\Lsp^2}}
\newcommand{\Linfsp}{{\Lsp^\infty}}
\newcommand{\LtRd}{{\Ltsp(\Rst^d)}}
\newcommand{\LtRtd}{{\Ltsp(\Rst^{2d})}}
\newcommand{\Nst}{{\mathbb N}}
\newcommand{\Gc}{{\mathcal{G}}}

\def\Esp{E}

\def\into{\hookrightarrow}

\newcommand{\Espd}{{E_d}}

\newcommand{\EspdLt}{{E_{d,\Ltsp}}}

\newcommand{\ip}[2]{\ensuremath{\left<#1,#2\right>}}
\newcommand{\sett}[1]{\ensuremath{\left \{ #1 \right \}}}
\newcommand{\abs}[1]{\ensuremath{\left| #1 \right| }}
\newcommand{\mes}[1]{\ensuremath{\left| #1 \right| }}

\newcommand{\iv}[1]{{#1}^\vee}
\newcommand{\rel}{\rho}

\newcommand{\rec}{S}

\newcommand{\recv}{\rec_U}
\newcommand{\win}{{W(L^\infty,L^1_w)}}
\newcommand{\winr}{{W_R(L^\infty,L^1_w)}}

\newcommand{\wrtd}{{W(L^\infty,L^1_w)(\Rtdst)}}

\newcommand{\wweak}{{W_R^{\rm weak}(L^\infty,L^1_w)}}
\newcommand{\wright}{{W_R(L^\infty,L^1_w)}}

\newcommand{\SEsp}{S_\Esp}
\newcommand{\SLt}{S_{L^2}}

\newcommand{\Spv}{S^p_v}

\newcommand{\Sjclass}{M^{\infty,1}_{\widetilde{w}}(\Rtdst)}
\newcommand{\consEw}{C_{\Esp,w}}

\newcommand{\Vstft}{V}

\newcommand{\multm}{M_m}
\newcommand{\multmg}{M_{\eta_\gamma}}
\newcommand{\multmgeps}{M^\varepsilon_{\eta_\gamma}}

\newcommand{\kernelenv}{K}
\newcommand{\env}{\Theta}
\newcommand{\molg}{{T_\gamma}}
\newcommand{\admolg}{{T^*_\gamma}}
\newcommand{\molgp}{{T_{\gamma'}}}
\newcommand{\phigk}{\phi^{\gamma}_k}
\newcommand{\lambdagk}{\lambda^{\gamma}_k}

\newcommand{\ellpv}{{\ell^p_v}}

\newcommand{\Lisp}{L^1}

\newcommand{\stft}{V_\varphi}

\newcommand{\tfmolg}{{H_\gamma}}

\newcommand{\locmg}{{H_{\eta_\gamma}}}

\newcommand{\locm}{{H_m}}
\newcommand{\locmgeps}{{H^\varepsilon_{\eta_\gamma}}}

\newcommand{\Mpv}{{M^p_v}}
\newcommand{\Miw}{{M^1_w}}
\newcommand{\Miwp}{M^\infty_{1/w}}

\newcommand{\coeft}{C_T}
\newcommand{\trace}{\operatorname{trace}}

\newtheorem{theorem}{Theorem}[section]
\newtheorem{lemma}[theorem]{Lemma}
\newtheorem{coro}[theorem]{Corollary}
\newtheorem{prop}[theorem]{Proposition}

\newtheorem{rem}[theorem]{Remark}
\newtheorem{definition}[theorem]{Definition}
\newtheorem{example}[theorem]{Example}
\theoremstyle{remark}

\title{Frames adapted to a phase-space cover}

\author{Monika D\"orfler} \email{monika.doerfler@univie.ac.at} 
\author{Jos\'e Luis Romero} \email{jose.luis.romero@univie.ac.at}
\thanks{Monika D\"orfler was supported by the Austrian Science Fund (FWF):[T384-N13] {\em Locatif} and by the WWTF project
{\em Audiominer} (MA09-24).}
\thanks{Jos\'e Luis Romero gratefully acknowledges support from the Austrian Science Fund (FWF): [M1586], [P22746-N13]
and [T384-N13].}
\address{Faculty of Mathematics, University of Vienna, Oskar-Morgenstern-Platz 1,A-1090 Wien, Austria}
\date{\today}


\begin{document}

\begin{abstract}
We construct frames adapted to a given cover of the time-frequency or time-scale plane. The main feature is that we allow
for quite general and possibly irregular covers. The frame members are obtained by maximizing their concentration in the
respective regions of phase-space. We present applications in time-frequency, wavelet and Gabor analysis.
\end{abstract}
\maketitle

\section{Introduction}
\label{sec_intro}
A time-frequency representation of a distribution $f \in \mathcal{S}'(\Rdst)$ is a function defined on $\Rdst \times \Rdst$
whose value at $z=(x,\xi)$ represents the influence of the frequency $\xi$ near $x$. The short-time Fourier transform (STFT)
is a standard choice for such a representation, popular in analysis and signal processing. It is defined, by means of
an adequate smooth and fast-decaying window function $\varphi \in \mathcal{S}(\Rdst)$, as 
\begin{align}
\label{eq_stft}
\Vstft_{\varphi}f(z)
=\int_{\mathbb{R}^d}f(t)\overline{\varphi(t-x)}e^{-2\pi i\xi t}dt,
\quad z=(x,\xi) \in \Rdst \times \Rdst.
\end{align}
If the window $\varphi$ is normalized by $\norm{\varphi}_2=1$, the distribution $f$ can be re-synthesized from its
time-frequency content by
\begin{align}
\label{eq_stft_inversion}
f(t) =\Vstft_{\varphi}^\ast\Vstft_{\varphi}f(t) = \int_{\Rdst \times \Rdst} \Vstft_{\varphi}f(x,\xi) \varphi(t-x)e^{2\pi i\xi t} dx d\xi,
\qquad t \in \Rdst.
\end{align}
This representation is extremely redundant. One of the aims of time-frequency analysis is to provide a representation of
an
arbitrary signal as a linear combination of elementary time-frequency atoms, which form a less redundant dictionary. The
standard choice is to let these atoms be time-frequency shifts of a single window function $\varphi$, thus providing a
uniform partition of the time-frequency plane. The resulting systems of atoms are known as {\em Gabor frames}. However,
in certain applications atomic decompositions adapted to a less regular pattern may be required (see for example
\cite{gr93-2,alcamo04-1,dadepe10,ro11}).

For example, a time-frequency partition may be derived from perceptual considerations. For audio signals, this means
that low frequency bins are given a finer resolution than bins in high regions, where better time-resolution is often
desirable, cp.~\cite{klsc10,dohove11}. Such a partition is schematically depicted in the left plot of 
Figure~\ref{Fig1.jpg}.

More irregular partitions may be desirable whenever the frequency characteristics of an analyzed signal change over time
and
require adaptation in both time \emph{and} frequency. For example, adaptive partitions obtained from information
theoretic criteria were suggested in~\cite{jato07,lirororo11}. In such a situation, partitions as irregular as shown in
the
right plot of  Figure~\ref{Fig1.jpg} can be appropriate.

In this article we consider the following problem. Given a - possibly irregular - cover of the time-frequency plane
$\Rtdst$,
we wish to construct a frame for $L^2(\Rdst)$ with atoms whose time-frequency concentration follows the shape of the
cover
members. This allows to vary the trade-off between time and frequency resolution along the time-frequency plane. The 
adapted frames are constructed by selecting, for each member of a given cover, a family of functions maximizing their
concentration in the corresponding region of the time-frequency domain, or phase-space. These functions can be obtained
as
eigenfunctions of time-frequency localization operators, as we now describe.

Given a bounded measurable set $\Omega \subseteq \Rtdst$ in the time-frequency plane, the \emph{time-frequency
localization operator}
$H_\Omega$ is defined by masking the coefficients in \eqref{eq_stft_inversion}, cf.~\cite{da88, da90},
\begin{align}
\label{eq_loc_op}
H_\Omega f(t) = \stft^*(1_\Omega\stft f)(t)=
\int_{\Omega} \Vstft_{\varphi}f(x,\xi) \varphi(t-x) e^{2\pi i\xi t} dx d\xi.
\end{align}
$H_\Omega$ is self-adjoint and trace-class, so we can consider its spectral decomposition 
\begin{align}
\label{eq_locop_spec}
H_{\Omega} f = \sum_{k=1}^\infty \lambda^\Omega_k \langle f , \phi_k^\Omega\rangle \phi_k^\Omega,
\end{align}
where the eigenvalues are indexed in descending order.
Note that $\ip{H_{\Omega} f}{f}=\ip{1_\Omega\stft f}{\stft f}=
\int_\Omega \abs{\stft f}^2$. Hence, the first eigenfunction of $H_{\Omega}$
is optimally concentrated inside $\Omega$ in the following sense
\begin{align*}
\int_\Omega \abs{\stft \phi^\Omega_1 (z)}^2\,dz =
\max_{\norm{f}_2=1} \int_\Omega \abs{\stft f(z)}^2\,dz.
\end{align*}

More generally, it follows from the Courant minimax principle, see e.g.~\cite[Section~95]{rina55}, that the first $N$ eigenfunctions of $H_{\Omega}$ form an orthonormal set in
$L^2(\mathbb{R}^d )$ that maximizes the quantity
$\sum_{j = 1}^N  \int_{\Omega} |\stft \phi^\Omega_j (z) |^2\,dz$ among all orthonormal sets of $N$ functions in
$L^2(\mathbb{R}^d )$. 
In this sense, their time-frequency
profile is optimally adapted to $\Omega$. Figure~\ref{Fig:EFboxes} illustrates this principle by showing some
time-frequency
boxes $\Omega$ along with the STFT and real part of the corresponding localization operator's first eigenfunctions.

\begin{figure}[tb]
\centerline{\includegraphics[width=0.6\textwidth]{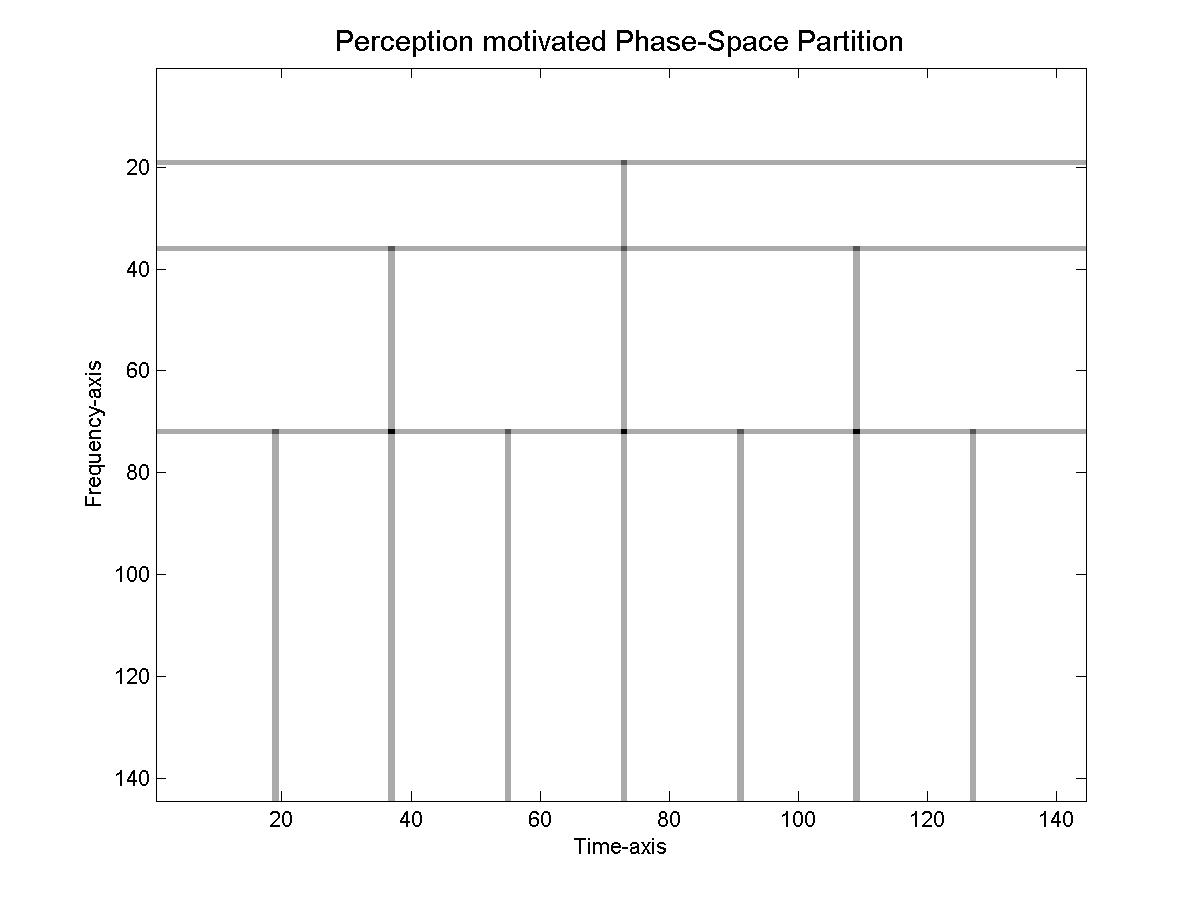}
 \includegraphics[width=0.6\textwidth]{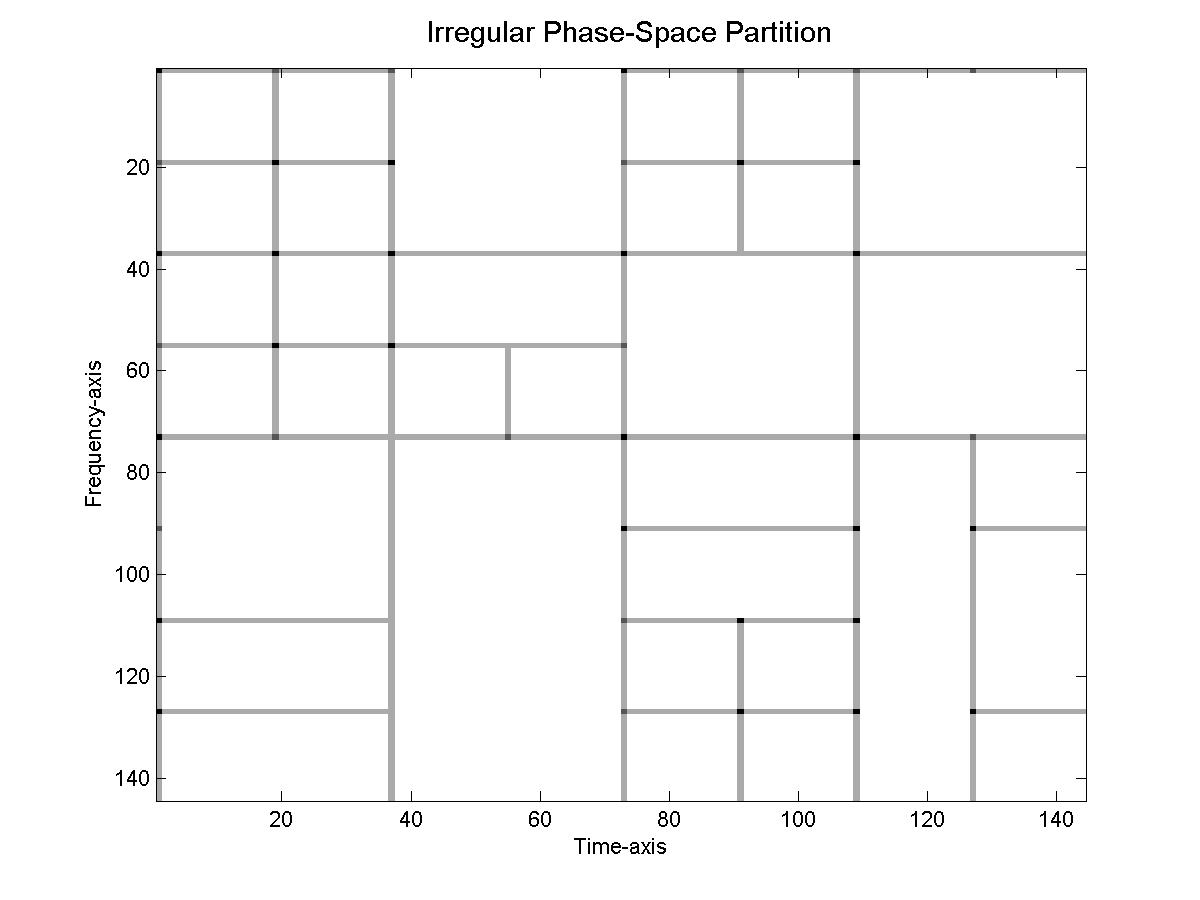}}
	\caption{Partitions in Time-Frequency}
	\label{Fig1.jpg}
\end{figure}
Based on these observations, we propose the following construction of frames. 
Let $\sett{\Omega_\gamma: \gamma \in \Gamma}$ be a cover of $\Rtdst$ consisting of bounded measurable sets.
In order to construct a frame adapted to the cover, we select, for each region $\Omega_\gamma$, the first $N_\gamma$
eigenfunctions
$\phi^1_{\Omega_\gamma}, \ldots, \phi^{N_\gamma}_{\Omega_\gamma}$ of the operator
$H_{\Omega_\gamma}$. We will prove that there is a number $\alpha>0$ such that
if $N_\gamma \geq \alpha \abs{\Omega_\gamma}$, then the collection of all the chosen eigenfunctions
spans $L^2(\Rdst)$ in a stable fashion. (Here $\abs{\Omega_\gamma}$ is the Lebesgue measure of $\Omega_\gamma$.)
Note that the condition $N_\gamma \geq \alpha \abs{\Omega_\gamma}$ is in accordance with the uncertainty principle,
which roughly says that for each time-frequency region $\Omega_\gamma$ there are only approximately
$\abs{\Omega_\gamma}$ degrees of freedom.

We allow for covers that are arbitrary in shape as long as they satisfy 
the following mild admissibility condition. An indexed set $\sett{\Omega_\gamma: \gamma \in \Gamma}$, is said to
be an \emph{admissible cover} of $\Rtdst$ if the following conditions hold.
\begin{itemize}
\item $\Gamma \subseteq \Rtdst$ and $\sup_{z \in \Rtdst} \#(\Gamma \cap B_1(z)) < +\infty$.
\item For each $\gamma \in \Gamma$, $\Omega_\gamma$ is a measurable subset of $\Rtdst$.
\item $\bigcup_{\gamma \in \Gamma} \Omega_\gamma = \Rtdst$.
\item There exists $R>0$ such that
\begin{align}
\label{eq_admissibility_no_inner}
\Omega_\gamma \subseteq B_R(\gamma), \qquad \gamma \in \Gamma.
\end{align}
\end{itemize}
Under this condition we prove the following.
\begin{theorem}
\label{th_sample}
Let $\sett{\Omega_\gamma: \gamma \in \Gamma}$ be an admissible cover of $\Rtdst$.
Then there exists a constant $\alpha>0$ such that for every choice of 
numbers $\{N_\gamma: \gamma \in \Gamma \} \subseteq \Nst$ satisfying
\begin{align*}
\alpha \abs{\Omega_\gamma} \leq N_\gamma \mbox{ and }
\sup_{\gamma \in \Gamma} N_\gamma < +\infty,
\end{align*}
the family of functions $\sett{\lambda^{\Omega_\gamma}_k \phi^{\Omega_\gamma}_k: \gamma \in \Gamma, 1 \leq k \leq
N_\gamma}$, obtained from the eigenfunctions and eigenvalues of the 
operators $H_{\Omega_\gamma}$ - cf. \eqref{eq_loc_op} and \eqref{eq_locop_spec} - 
is a frame of $L^2(\Rdst)$. That is, for some constants $0 < A \leq B < + \infty$,
the following frame inequality holds
\begin{align}
\label{eq_frame}
A \norm{f}_2^2
\leq
\sum_{\gamma \in \Gamma} \sum_{k=1}^{N_\gamma} \abs{\ip{f}{\lambda^{\Omega_\gamma}_k\phi^{\Omega_\gamma}_k}}^2
\leq
B \norm{f}_2^2,
\qquad f \in L^2(\Rdst).
\end{align}
\end{theorem}
Theorem \ref{th_sample} is proved at the end of  Section \ref{Sec:FoE}. We also show that if an admissible cover
$\sett{\Omega_\gamma: \gamma \in \Gamma}$ satisfies the additional {\em inner regularity condition}:
\begin{itemize}
\item There exist $r>0$ such that 
\begin{align}
\label{eq_admissibility}
B_r(\gamma) \subseteq \Omega_\gamma, \qquad \gamma \in \Gamma,
\end{align}
\end{itemize}
then the statement of Theorem \ref{th_sample} remains valid, if  the functions $\lambda^{\Omega_\gamma}_k\phi^{\Omega_\gamma}_k$ are replaced by their unweighted versions $\phi^{\Omega_\gamma}_k$ (see
Theorem \ref{th_frames_tf_noeigen}). In this case $\inf_\gamma \abs{\Omega_\gamma} \geq \abs{B_\gamma(0)}$, so
the lower bound on $N_\gamma$ in the hypothesis of Theorem \ref{th_sample}
can be expressed as $N_\gamma \geq \tilde{\alpha}$, for some constant $\tilde{\alpha}>0$.

While Theorem~\ref{th_sample} was our main motivation,  it is just a sample of our results. The introduction of an abstract
model for
phase space provides sufficient flexibility to obtain  variants of Theorem~\ref{th_sample} in the
context of time-scale analysis (Theorem~\ref{Th:WavCase} ) and of discrete time-frequency representations (Theorem~\ref{Th:GabMul} ).

\begin{figure}[tb]
\centerline{\includegraphics[width = 1.05\textwidth]{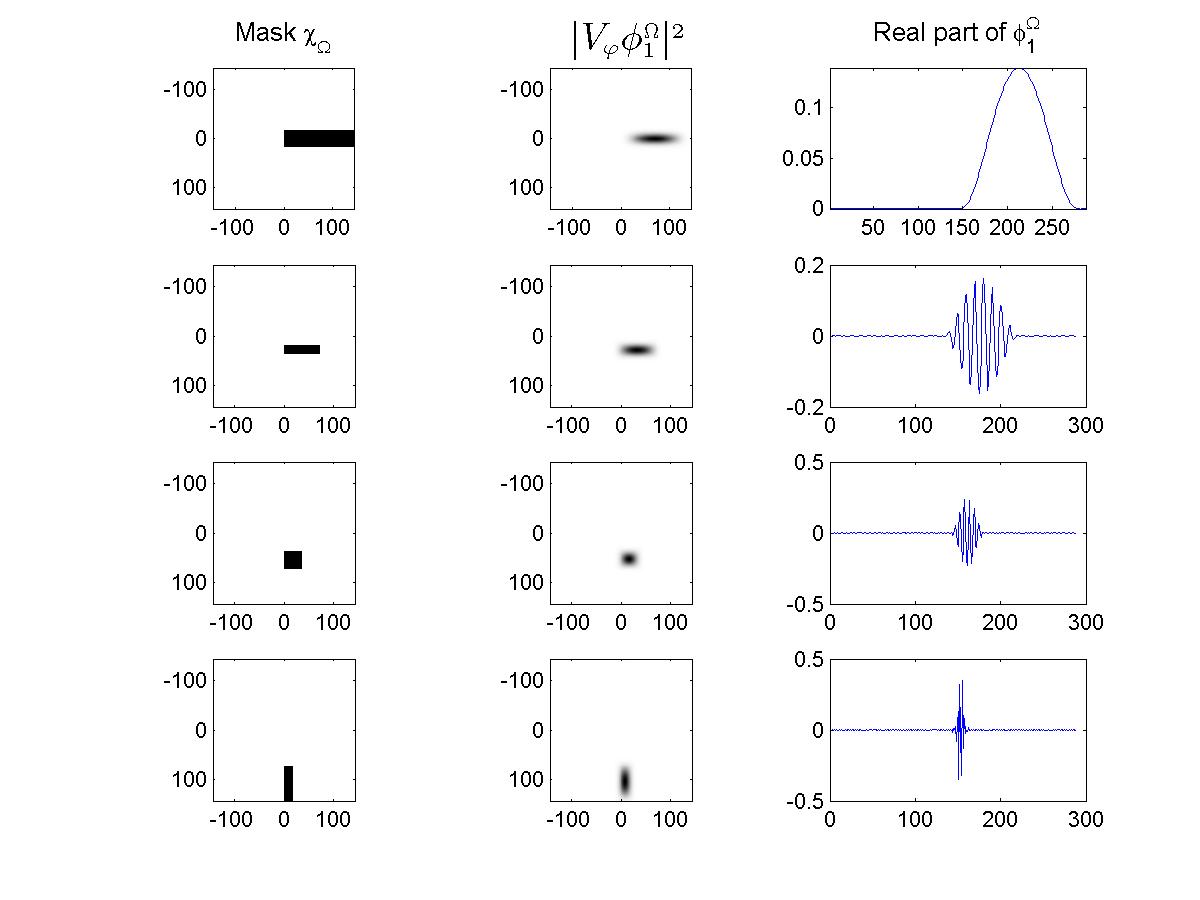}}
	\caption{Four different rectangular masks in time-frequency domain and the first eigenfunctions of the
corresponding localization operators. Middle plots show the absolute value squared of the STFT and right plots show the
real part.}
	\label{Fig:EFboxes}
\end{figure}

\subsection{Technical overview}
The proofs of the main results in this paper are based on two major observations. First,
the norm equivalence\footnote{For two non-negative functions $f,g: X \to (0,+\infty)$,
the statement $f \approx g$ means that there exist constants
$c,C \in(0,+\infty)$ such that $cf(x) \leq g(x) \leq Cf(x)$ for all $x \in X$.}
\begin{align}
\label{eq_intro_norm_equiv}
\norm{f}_2^2 \approx
\sum_{\gamma\in\Gamma} \norm{H_{\eta_\gamma} f}_2^2,
\qquad f \in \LtRd,
\end{align}
holds for a family of time-frequency localization operators
\begin{align}
\label{eq_intro_tfloc}
H_{\eta_\gamma} f(t) = \int_{\Rdst\times\Rdst}
\eta_\gamma(x,\xi) \Vstft_{\varphi}f(x,\xi) \varphi(t-x) e^{2\pi i\xi t} dx d\xi,
\end{align}
provided that the symbols $\eta_\gamma: \Rtdst \to [0,+\infty)$ satisfy
\begin{equation}\label{eq_intro_regular_symbols} 
\sum_{\gamma\in\Gamma} \eta_\gamma(z) \approx 1
\end{equation}
and the enveloping condition
\begin{align}
\label{eq_intro_envelope}
\eta_\gamma(z) \leq g(z-\gamma),
\mbox{ for some $g \in \Lisp(\Rtdst)$ and $\gamma \in \Gamma$, with $\Gamma \subseteq \Rtdst$ a lattice.} 
\end{align}
The inequalities \eqref{eq_intro_norm_equiv} were first proved in \cite{dofegr06} for symbols of the form
$\eta_\gamma(z)=g(z-\gamma) $
and $\Gamma=\Zst^{2d}$, then for a general lattice in \cite{dogr11}, and finally for fully irregular symbols
satisfying \eqref{eq_intro_envelope} in \cite{ro12}. It is interesting to note that the proofs in \cite{dofegr06,
dogr11}
are \emph{based} on the
observation that under condition \eqref{eq_intro_regular_symbols} the norm-equivalence \eqref{eq_intro_norm_equiv}
is equivalent to the fact that finitely many eigenfunctions of the operator $H_g$ generate a multi-window Gabor
frame over the lattice $\Gamma$. The proof of the general case in \cite{ro12} does not explicitly involve the
eigenfunctions
of the operators $H_{\eta_\gamma}$ nor does it rely on tools specific to Gabor frames
on lattices. Thus the question arises
whether it is also possible in the irregular case  to construct a frame consisting of finite sets of eigenfunctions of the operators
$H_{\eta_\gamma}$. Here, this question is given a positive answer.

Second, the observation that \eqref{eq_intro_norm_equiv} remains valid when the operators $H_{\eta_\gamma}$
are replaced by finite rank approximations $H^\varepsilon_{\eta_\gamma}$ obtained by thresholding their
eigenvalues,~cf.~Theorem~\ref{th_finite_rank_tf}, is the core of the proof of our main results. This
finite rank approximation is in turn achieved by proving that the operators $H_{\eta_\gamma}$ behave ``globally'' like
projectors. More precisely, in Proposition~\ref{prop_squared_covers}
we obtain the following extension of \eqref{eq_intro_norm_equiv}
\begin{align}
\label{eq_intro_norm_equiv1}
\norm{f}_2^2 \approx
\sum_{\gamma\in\Gamma} \norm{H_{\eta_\gamma} f}_2^2
\approx
\sum_{\gamma\in\Gamma} \norm{(H_{\eta_\gamma})^2 f}_2^2.
\end{align}
This will allow us to ``localize in phase-space'' the $L^2$-norm estimates relating $H^\varepsilon_{\eta_\gamma}$ 
and $H_{\eta_\gamma}$.

Note that, in general, the operators $H_{\eta_\gamma}$ have infinite rank even if $\eta_\gamma$ is the characteristic
function of a compact set (see Lemma \ref{lemma_loc_infinite_rank}). Consequently $\norm{H_{\eta_\gamma} f}_2
\not\approx
\norm{(H_{\eta_\gamma})^2 f}_2$ and therefore the global properties of the
family $\sett{(H_{\eta_\gamma})^2: \gamma \in \Gamma}$ are crucial to prove \eqref{eq_intro_norm_equiv1}. While the
squared
operators
$(H_{\eta_\gamma})^2$ are not time-frequency localization operators, their time-frequency localizing
behavior is preserved under conditions as given by  \eqref{eq_intro_envelope} and they are the prototypical example of
a family of operators that is \emph{well-spread in the time-frequency plane}. The latter notion is defined in
Section~\ref{sec_abs_mol} and we exploit the fact that the tools from \cite{ro12} are valid for these operator
families.

For clarity, we choose to accentuate  the case of time-frequency analysis but all the proofs are carried out in an
abstract
setting that yields, for example, analogous consequences in time-scale analysis.

\subsection{Organization}
The article is organized as follows. Section \ref{sec_bla} motivates the abstract model to study phase-space covers
and gives the main examples to keep in mind. Section~\ref{sec_abstract} formally introduces the abstract model for
phase-space and Section~\ref{sec_abs_mol} presents certain key technical notions, in particular the properties required
for a family of localization operators to exhibit an almost-orthogonality property. In Section~\ref{sec_tf}, we first
prove our results in the context of time-frequency analysis, where some technical problems of the abstract setting do
not arise. In addition, in the context of time-frequency analysis, we are able to extend the result on phase-space
adapted frames from $\LtRd$ to an entire class of Banach spaces, the modulation spaces, by exploiting
spectral invariance results for pseudodifferential operators.  Theorem~\ref{th_frames_tf} in Section~\ref{sec_frames_tf}
is the general version of Theorem~\ref{th_sample} stated above. This latter result in proved after
Theorem~\ref{th_frames_tf} as an application. Section~\ref{sec_abs_frames} develops the results 
in the abstract setting. These are then applied to time-scale analysis in Section~\ref{sec_wavelets}. Finally, Section
\ref{sec_gabor} contains an additional application of the abstract results to time-frequency analysis, this
time using Gabor multipliers, which are time-frequency masking operators related to a discrete time-frequency
representation (Gabor frame). The atoms thus obtained  maximize their time-frequency concentration with respect to a
weight on a discrete time-frequency grid and the resulting frames are relevant in numerical applications.

For clarity, the presentation of the results highlights the case of time-frequency analysis, which was our main
motivation. Most of the technicalities in Section \ref{sec_abstract} are irrelevant to that setting (although they are
relevant for time-scale analysis). The reader interested mainly in time-frequency analysis is encouraged to jump
directly to Section \ref{sec_tf} and then go back to Sections \ref{sec_abstract} and \ref{sec_abs_mol} having a clear
example in mind.

The article also contains two appendices providing auxiliary results related to the almost-orthogonality tools from
\cite{ro12} and spectral invariance of pseudodifferential operators.
\section{Phase-space}\label{sec_bla}
In this section, we introduce an abstract model for phase space. We first provide some motivation and the main examples
to keep in mind.
\subsection{The time-frequency plane as an example of phase-space}
\label{sec_bla_1}
In time-frequency analysis,
a distribution $f \in \mathcal{S}'(\Rdst)$ is studied by means of a time-frequency representation,
e.g. the STFT $V_\varphi f: \Rdst \times \Rdst \to \bC$, cf.~\eqref{eq_stft}.
In signal processing, where $f$ is called a signal, the domain $\Rdst$ is referred to as \emph{signal space} while
$\Rdst \times \Rdst$ is referred to as \emph{phase-space}. This terminology borrows some intuition from mechanics, where
the position of a freely moving particle is described by a point $x$ in the configuration space $\Rdst$, while
$(x,\xi)$ describes a pair of position-momentum variables, belonging to the phase-space $\Rdst \times \Rdst$.

Since $f$ can be resynthesized from $V_\varphi f$ - cf. \eqref{eq_stft_inversion} - all its properties 
can in principle be reformulated as properties of $V_\varphi f$. However, not every function on $\Rdst \times \Rdst$ is
the
STFT of a distribution on $\Rdst$. Indeed, if we let $S := V_\varphi \LtRd$ be the image of
$\LtRd$ under $V_\varphi$, it turns out that $S$ is a \emph{reproducing kernel space}
(see \cite[Chapters 3 and 11]{gr01}). This means that $S$ is a
closed subspace of $\LtRtd$ consisting of continuous functions and that  for all $(x,\xi) \in \Rdst \times
\Rdst$ the evaluation functional $S \ni F \mapsto F(x,\xi) \in \bC$ is continuous. In particular
$S \varsubsetneq \LtRd$.

The fact that $S := V_\varphi \LtRd$ is a ``small'' subspace of $\LtRd$ is important to understand the problem studied
in this article. In designing a frame for $\LtRd$ with a prescribed phase-space profile, the challenge lies in
the fact that the shapes we design in $\Rtdst$ must correspond to functions in the small subspace $S$. The role of 
time-frequency localization operators is crucial and will be detailed in the rest of this section.

For $m \in L^\infty(\Rdst \times \Rdst)$, the \emph{time-frequency localization operator} with symbol $m$
is
\begin{align}
\label{eq_loc_op_2}
H_m f(t) = \int_{\Rdst \times \Rdst} \Vstft_{\varphi}f(x,\xi) m(x,\xi) \varphi(t-x) e^{2\pi i\xi t} dx d\xi,
\qquad t \in \Rdst.
\end{align}
If  $m=1_\Omega$ is the characteristic function of a set $\Omega$ we write $H_{\Omega}$ instead of $H_{1_\Omega}$,~cf.
\eqref{eq_loc_op}.

The STFT (with respect to a fixed normalized window $\varphi \in \LtRd$) defines a map $V_\varphi: \LtRd
\to \LtRtd$. Its adjoint $V^*_\varphi:\LtRtd \to \LtRd$ is given by
\begin{align*}
V^*_\varphi F(t) = \int_{\Rdst \times \Rdst} F(x,\xi) \varphi(t-x)e^{2\pi i\xi t} dx d\xi,\qquad t\in\Rdst.
\end{align*}
The inversion formula in \eqref{eq_stft_inversion} says that $V_\varphi^*V_\varphi = I_{\LtRd}$. Hence $V_\varphi$ is an
isometry
on $\LtRd$. Consider a symbol $m \in L^\infty(\Rdst \times \Rdst)$ and the time-frequency localization operator from
\eqref{eq_loc_op_2}. Using $V_\varphi$ and $V^*_\varphi$, the definition reads
\begin{align}
\label{eq_toep_1}
H_m f = V^*_\varphi (m V_\varphi f), \qquad f \in \LtRd.
\end{align}
The fact that $V_\varphi$ is an isometry with range $S$ implies that  $V_\varphi V_\varphi^* = P_S$ is the
orthogonal projection $\LtRtd \to S$. Explicitly, for $F \in L^2(\Rdst\times\Rdst)$
\begin{align}
\label{eq_form_ps}
P_S F(x',\xi') = \int_{\Rdst \times \Rdst} F(x,\xi) 
\ip{\varphi(\cdot-x)e^{2\pi i\xi \cdot}}{\varphi(\cdot-x')e^{2\pi i\xi' \cdot}}_{L^2(\Rdst)}
dx d\xi.
\end{align}
It follows from \eqref{eq_toep_1} that
\begin{align*}
(V_\varphi H_m V^*_\varphi) V_\varphi f = P_S (m V_\varphi f), \qquad f \in \LtRd.
\end{align*}
Since $F=V_\varphi f$ is the generic form of a function in $S$, we obtain that
\begin{align*}
(V_\varphi H_m V^*_\varphi) F = P_S (m F), \qquad F \in S.
\end{align*}
This means that the time-frequency localization operator $H_\Omega$ is unitarily equivalent 
to the operator $M_m: S \to S$ given by
\begin{align}
\label{eq_toep_2}
M_m F := P_S (m F), \qquad F \in S.
\end{align}
The operator $M_m$ consists of multiplication by $m$ followed by projection onto $S$. 
We will call these operators \emph{phase-space} multipliers: they apply a mask $m$ to
a function $F=V_\varphi f \in S$, typically yielding a function $m\cdot F\notin S$, and then provide the best $L^2$
approximation of $m\cdot F$ within $S$. (These operators are sometimes also called Toeplitz operators for the STFT.)

\subsection{Other transforms}
\label{sec_bla_2}
The interpretation of time-frequency localization operators as phase-space multipliers (multiplication followed by
projection) is central to this article. We will consider a general setting where the role of the STFT can be replaced by
other transforms. An important example is the wavelet transform of a function $L^2(\Rdst)$
with respect to an adequate window $\psi \in \mathcal{S}(\Rdst)$,
\begin{align}
\label{eq_wav_trans}
W_\psi f(x,s) = s^{-d/2}\int_{\Rdst} f(t)
\overline{\psi \left(\frac{t-x}{s}\right)}dt,
\qquad x \in \Rdst, s>0.
\end{align} 
(See Section \ref{sec_wavelets} for details.) In analogy with time-frequency analysis,
we still call $\Rdst \times \Rst_+$ the phase-space associated with the wavelet transform. This terminology is
justified by the fact that there is a formal analogy between the two contexts. The range of the wavelet transform
$W_\psi \LtRd$ is, under suitable assumptions on $\psi$, a reproducing kernel subspace of
$L^2(\Rdst \times \Rst_+, s^{-(d+1)}dx ds)$
and time-scale localization operators are defined in analogy to the time-frequency localization operators (see
Section \ref{sec_wavelets} for explicit formulas).

The  similarity between the time-frequency and time-scale contexts stems from the fact that both are
associated with representations of a locally-compact group. In the former case, the Heisenberg group
acts on $\LtRd$ by translations and modulations, while in the latter, the affine group acts by translations and
dilations. The theory of \emph{coorbit spaces}~\cite{fegr89,fegr89-1} treats this situation in general, studying the transform
associated with the representation coefficients of a group action and associating a range of function spaces to it. The
model for abstract phase space to be introduced in Section~\ref{sec_abstract} is largely inspired by \cite{fegr89}. It 
allows for the simultaneous treatment of various settings, since it makes no explicit reference to an integral
transform. The main ingredients are a ``big'' space $E$ 
called the environment and a ``small'' subspace called the atomic space. In the  time-frequency example  $E=\LtRtd$ and $S=V_\varphi \LtRd$. Similarly, the case of time-scale analysis
(wavelets) uses $E=L^2(\Rdst \times \Rst_+, s^{-(d+1)}dx ds)$ and $S=W_\psi \LtRd$.
Building on the intuition
provided by time-frequency and time-scale analysis, we think of $S$ as a collection
of phase-space representations for functions in $\LtRd$, while the environment $E$ is big enough for certain
operations, such as pointwise multiplication by arbitrary bounded measurable functions, 
to be well-defined. One central assumption of the model in Section \ref{sec_abstract} is then the existence of a 
projector $P:E\to S$, so that one can consider phase-space multipliers like in \eqref{eq_toep_2}.

The price to pay for this unified approach is a certain level of technicality. The Euclidean space is not
suitable any more as a model for the domain of the functions in $S$: we need to consider a general locally-compact
group. To provide an easily accessible example,  the main concepts  will be spelled out in
the concrete case of time-frequency analysis.

\subsection{Abstract phase-space}
\label{sec_abstract}
\subsubsection{Locally compact groups and function spaces}
Throughout the article $\Gc$ will be a locally compact, $\sigma$-compact,
topological group with modular function $\Delta$. 
The left Haar measure of a set $X \subseteq \Gc$ will be denoted by
$\mes{X}$. Integration will always be considered with respect to the left Haar
measure.
For $x \in \Gc$, we denote by $L_x$ and $R_x$ the operators of
left and right translation, defined by
$L_x f(y) = f(x^{-1}y)$ and $R_x f(y) = f(yx)$.
We also consider the involution $\iv{f}(x) = f(x^{-1})$.

Given two non-negative functions $f,g$ we write $f \lesssim g$ if there
exists a constant $C \geq 0$ such that $f \leq C g$. We say that $f \approx
g$ if both $f \lesssim g$ and $g \lesssim f$.
The characteristic function of a set $A$ will be denoted by $1_A$.

A set $\Gamma \subseteq \Gc$ is called \emph{relatively separated} if for some (or any) $V \subseteq \Gc$, relatively
compact neighborhood of the identity, the quantity - called the \emph{spreadness of $\Gamma$} -
\begin{align}
\label{eq_spread}
\rel(\Gamma) = \rel_V(\Gamma)
:= \sup_{x \in \Gc} \# (\Gamma \cap x V)
\end{align}
is finite, i.e. if the amount of elements of $\Gamma$ that lie in any left
translate of $V$ is uniformly bounded.

The following definition introduces a class of function spaces on $\Gc$. The
Lebesgue spaces $L^p(\Gc)$ are natural examples.
\begin{definition}[Banach function spaces] A Banach space $\Esp$
is called a solid, translation invariant BF space if it satisfies the following.
\begin{itemize}
\item[(i)] $\Esp$ is continuously embedded into $L_{\mathrm{loc}}^1(\Gc)$,
the space of complex-valued locally integrable functions on $\Gc$.

\item[(ii)] Whenever $f \in \Esp$ and $g: \Gc \to \bC$ is a measurable function 
such that $\abs{g(x)} \leq \abs{f(x)}$ a.e., it is true that $g \in \Esp$
and $\norm{g}_\Esp \leq \norm{f}_\Esp$.

\item[(iii)] $\Esp$ is closed under left and right translations
(i.e. $L_x \Esp \subseteq \Esp$ and $R_x \Esp \subseteq \Esp$, for all $x \in
\Gc$) and the following relations hold with the corresponding norm estimates
\begin{align}
\label{conv_mod_assumption}
L^1_u(\Gc)*\Esp \subseteq \Esp
\mbox{  and  }
\Esp*L^1_v(\Gc) \subseteq \Esp,
\end{align}
where $u(x) := \norm{L_x}_{\Esp \to \Esp}$,
$v(x) := \Delta(x^{-1})\norm{R_{x^{-1}}}_{\Esp \to \Esp}$.
\end{itemize}
\end{definition}
\begin{definition}[Admissible weights for a BF space]
\label{def_adm_for}
Given a solid, translation invariant BF space $\Esp$, a function $w: \Gc \to (0,+\infty)$ satisfying 
\begin{align}
\label{weight_w_delta}
&w(x) = \Delta(x^{-1})w(x^{-1}),
\\
\label{weight_w_submult}
&w(xy) \leq w(x)w(y)
\mbox{ (submultiplicativity)},
\\
\label{weight_w_admissible}
&w(x) \geq \consEw \max \sett{u(x), \; u(x^{-1}), \; v(x),
\; \Delta(x^{-1})v(x^{-1})},
\\
\nonumber
&\qquad \mbox{where } u(x) := \norm{L_x}_{\Esp \to \Esp},\, v(x) := \Delta(x^{-1})\norm{R_{x^{-1}}}_{\Esp \to \Esp},
\end{align}
for some constant $\consEw>0$ is called an \emph{admissible weight} for $\Esp$.
\end{definition}
If $w$ is admissible for $\Esp$, it follows that $w(x) \gtrsim 1$,
$L^1_w * \Esp \subseteq \Esp$
and $\Esp*L^1_w \subseteq \Esp$, and  the constants in the corresponding norm estimates  depend only on $\consEw$,~cf.~\cite{fegr89}.

For a solid translation invariant BF space $\Esp$ and 
$\Gamma \subseteq \Gc$ a relatively separated set,
we construct discrete versions $\Espd$ as follows. 
Fix $V$, a symmetric relatively compact neighborhood of the identity and
let
\begin{align*}
\Espd = \Espd(\Gamma) := \set{c \in {\bC}^\Gamma}
{\sum_{\gamma\in \Gamma} \abs{c_\gamma} 1_{\gamma V} \in \Esp},\mbox{ with norm }
 \norm{\left(c_\gamma \right)_{\gamma \in \Gamma}}_{\Espd}
 := \Bignorm{\sum_{\gamma\in \Gamma} \abs{c_\gamma} 1_{\gamma V}}_\Esp.
\end{align*}
The definition depends on $V$, but a different choice of $V$ yields the same space with equivalent norm
(this is a consequence of the right invariance of $\Esp$,
see for example \cite[Lemma 2.2]{ra07-4}).
For $\Esp=L^p_w$, the corresponding discrete space $\Espd(\Gamma)$ is just $\ell^p_w(\Gamma)$, where the (admissible)
weight $w$ is restricted to the set $\Gamma$.

We next define the left \emph{Wiener amalgam space} with respect to a solid, translation invariant BF space $\Esp$.
Let $V$ be again a symmetric, relatively compact neighborhood of the identity. For a locally
bounded
function $f: \Gc \to \bC$ consider the left \emph{local maximum function} defined by
\begin{align*}
f^\# (x) := \supess_{y \in V} \abs{f(xy)}
=\norm{f \cdot (L_x 1_V)}_\infty
, \qquad x \in \Gc,
\end{align*}
and similarly the right local maximum function is given by 
$f_\# (x) := \supess_{y \in V} \abs{f(yx)} = \norm{f \cdot (R_x 1_V)}_\infty$.
\begin{definition}[Wiener amalgam spaces]
\begin{align*}
W(\Linfsp, \Esp)(\Gc) := \set{f \in L^\infty_{\textit{loc}}(\Gc)}{f^\# \in \Esp},
\end{align*}
with norm $\norm{f}_{W(\Linfsp, \Esp)} := \norm{f^\#}_\Esp$. 
The right Wiener amalgam space $W_R(\Linfsp, \Esp)$ is defined similarly,
this time using the norm $\norm{f}_{W_R(\Linfsp, \Esp)} := \norm{f_\#}_\Esp$.
\end{definition}
A different choice of $V$ yields
the same spaces with equivalent norms (see for example \cite[Theorem 1]{fe83} or \cite{fe90}).
When $\Esp$ is a weighted $L^p$ space on $\Rdst$, the corresponding amalgam space coincides with the classical
$L^\infty-\ell^p$ amalgam space \cite{ho75, fost85}. In the present article we will be mainly interested in the spaces
$W(L^\infty,L^1_w)$ and
$W_R(L^\infty,L^1_w)$ for which we  need the following facts.
\begin{prop}
\label{prop_conv_algebra}
The spaces $W(L^\infty,L^1_w)$ and $W_R(L^\infty,L^1_w)$ are convolution algebras. That is,
the relations $W(L^\infty,L^1_w) * W(L^\infty,L^1_w) \into W(L^\infty,L^1_w)$ and
$W_R(L^\infty,L^1_w) * W_R(L^\infty,L^1_w) \into W_R(L^\infty,L^1_w)$ hold together with the
corresponding norm estimates.
\end{prop}
\begin{proof}
The left amalgam space satisfies the (translation invariance)
relation $L^1_w * W(L^\infty,L^1_w) \into W(L^\infty,L^1_w)$. This is a particular case of \cite[Theorem
7.1]{fegr89-1}
and can also be readily deduced from the definitions. Since $W(L^\infty,L^1_w) \into L^1_w$, the statement about
$W(L^\infty,L^1_w)$ follows. The involution $\iv{f}(x) = f(x^{-1})$ maps $W_R(L^\infty,L^1_w)$ isometrically onto
$W(L^\infty,L^1_w)$ (because $V=V^{-1}$) and satisfies $\iv{(f*g)}=\iv{g}*\iv{f}$. Hence the statement about the right
amalgam space follows from the one about the left one.
\end{proof}

\subsubsection{The model for abstract phase-space}
\label{sec_model}
In the abstract model for phase-space we consider a solid BF space $\Esp$ (called the environment) and a certain
distinguished subspace $\SEsp$,  which is the range of an idempotent operator $P$. The precise form of the model is
taken from \cite{ro12} and is designed to fit the theory in \cite{fegr89} (see also \cite{dastte04-1,nasu10}).
We list a number of ingredients in the form of two assumptions: (A1) and (A2).
\begin{itemize}
\item[(A1)]
\begin{itemize}
\item  $\Esp$ is a solid, translation invariant BF space,
called \emph{the environment}.
\item $w$ is an admissible weight for $\Esp$.
\item $\SEsp$ is a closed complemented subspace of $\Esp$, called \emph{the atomic subspace}.
\item Each function in $\SEsp$ is continuous.
\end{itemize}
\item[(A2)] $P$ is an operator and $\kernelenv$ is a non-negative function 
satisfying the following.
\begin{itemize}
\item $P: W(L^1,L^\infty_{1/w}) \to L^\infty_{1/w}$ is
a (bounded) linear operator,
\item $P(\Esp) = \SEsp$ and $P(f) = f, \mbox{for all } f \in \SEsp$,
\item $\kernelenv \in W(L^\infty,L^1_w) \cap W_R(L^\infty,L^1_w)$,
\item For $f \in W(L^1,L^\infty_{1/w})$,
\begin{align}
\label{eq_P_dominated}
\abs{P(f)(x)} \leq \int_{\Gc} \abs{f(y)} \kernelenv(y^{-1}x) dy,
\qquad x \in \Gc.
\end{align}
\end{itemize}
\end{itemize}
When $\Esp=L^2(\Gc)$ we will additionally assume the following.
\begin{itemize}
\item[(A3)] $P: L^2(\Gc) \to \SLt$ is the orthogonal projection.
\end{itemize}
Note that Assumption (A2) means that the retraction $\Esp \to \SEsp$ is given by an operator
dominated by right convolution with a kernel in $W(L^\infty,L^1_w) \cap W_R(L^\infty,L^1_w)$.
\begin{example}
\label{ex_tf}
The discussion in Section \ref{sec_bla_1} provides the main example of the abstract model. If $\varphi \in
\mathcal{S}(\Rdst)$ we can let $\Gc=\Rdst \times \Rdst$, $\Esp = \LtRtd$ and $S=\SEsp = V_\varphi \LtRd$.
For this choice $w \equiv 1$ is an admissible weight. We also let
$P=P_S: \LtRtd \to \SEsp$ be the orthogonal projection. To see that $P$ satisfies (A2), note that from
\eqref{eq_form_ps}
\begin{align*}
\abs{P_S F(x',\xi')} &\leq
\int_{\Rdst \times \Rdst} \abs{F(x,\xi)}
\abs{\ip{\varphi(\cdot-x)e^{2\pi i\xi \cdot}}{\varphi(\cdot-x')e^{2\pi i\xi' \cdot}}}
dx d\xi
\\
&=
\int_{\Rdst \times \Rdst} \abs{F(x,\xi)} \abs{V_\varphi \varphi(x'-x,\xi'-\xi)} dx d\xi.
\end{align*}
It is easy to see that the fact that $\varphi \in \mathcal{S}(\Rdst)$ implies
that $K:=\abs{V_\varphi \varphi} \in W(L^\infty,L^1)(\Rtdst)$.
\end{example}

For the remainder of Section \ref{sec_abstract}, we assume (A1) and (A2). Under these conditions the following holds.
\begin{prop}\cite[Proposition 3]{ro12}
\label{prop_P_into_am}
\mbox{}

\begin{itemize}
\item[(a)] $P$ boundedly maps $\Esp$ into $W(L^\infty,\Esp)$.
\item[(b)] $\SEsp \hookrightarrow W(L^\infty, \Esp)$.
\item[(c)] If $f \in W(L^1,L^\infty_{1/w})$, then
$\norm{P(f)}_{L^\infty_{1/w}} \lesssim
\norm{f}_{W(L^1,L^\infty_{1/w})} \norm{\kernelenv}_{W_R(L^\infty,L^1_w)}$.
\item[(d)] If $f \in W(L^1,L^\infty)$, then
$\norm{P(f)}_{L^\infty} \lesssim
\norm{f}_{W(L^1,L^\infty)} \norm{\kernelenv}_{W_R(L^\infty,L^1_w)}$.
\end{itemize}
\end{prop}
\begin{rem}
Since $w \gtrsim 1$, $L^\infty \into L^\infty_{1/w}$.
\end{rem}
\begin{rem}
\label{rem_setting_unif}
The estimates in Proposition \ref{prop_P_into_am} hold uniformly for all the
spaces $\Esp$ with the same weight $w$ and the same constant $\consEw$ (cf.
\eqref{weight_w_admissible}). This is relevant to the applications, where
the same projection $P$ is used with different spaces $\Esp$ and corresponding subspaces
$\SEsp=P(\Esp)$.
\end{rem}

\subsubsection{Phase-space multipliers}
\label{sec_phmul}
Recall the projector $P:\Esp \to \SEsp$ from (A2). For $m \in L^\infty(\Gc)$, the \emph{phase-space multiplier}
with \emph{symbol} $m$ is the operator $M_m: \SEsp \to \SEsp$ defined by
\begin{align*}
\multm(f) := P(mf), \qquad f \in \SEsp.
\end{align*}
It is bounded by Proposition~\ref{prop_P_into_am} and the solidity of $\Esp$:
\begin{align}
\label{eq_bound_M}
\norm{\multm(f)}_\Esp \lesssim \norm{m}_\infty \norm{f}_\Esp,
\quad f\in\SEsp.
\end{align}
\begin{example}
As discussed in Section \ref{sec_bla_1}, time-frequency localization operators \eqref{eq_loc_op_2} are
unitarily equivalent via the STFT to phase-space multipliers, with $\Esp$, $\SEsp$ and $P$ as in Example
\ref{ex_tf}.

In the context of time-scale analysis, one can define operators in analogy to time-frequency
localization operators  using the wavelet transform in \eqref{eq_wav_trans} instead of the STFT. These are called time-scale localization operators or wavelet multipliers \cite{dapa88, wo99, limowo08}. With
an adequate choice of $\Esp$, $\SEsp$ and $P$, time-scale localization operators are unitarily equivalent to phase-space
multipliers. This is developed in Section \ref{sec_wavelets}.
\end{example}
For future reference we note some Hilbert-space properties of phase-space multipliers
(when $\Esp=\Ltsp(\Gc)$), which are well-known for time-frequency localization operators, cf.~\cite{boco02,feno03,wo02,
boto08}.

\begin{prop}
\label{prop_mult_hilbert}
Let $\Esp=\Ltsp(\Gc)$ and assume (A1), (A2) and (A3). Then the following hold.
\begin{itemize}
\item[(a)] Let $m_1, m_2 \in \Linfsp(\Gc)$ be real-valued.
If $m_1 \leq m_2$ a.e., then $M_{m_1} \leq M_{m_2}$ as operators.
In particular if $m$ is non-negative and bounded, then $\multm$ is a positive operator.

\item[(b)] Let $m \in \Lisp(\Gc) \cap \Linfsp(\Gc)$ be non-negative. Then $\multm:\SLt \to \SLt$ is trace-class
and $\trace(\multm) \lesssim \norm{m}_1$.
\end{itemize}
\end{prop}
\begin{proof}
To prove (a), let $f \in \SLt$ and note that by (A3),
\begin{align*}
\ip{M_{m_1} f}{f} &= \ip{P(m_1 f)}{f} =\ip{m_1 f}{f} 
=\int_\Gc m_1(x) \abs{f(x)}^2 \, dx
\\
&\leq \int_\Gc m_2(x) \abs{f(x)}^2 \, dx
=\ip{M_{m_2} f}{f}.
\end{align*}
Let us now prove (b). For $f \in \SLt$ and $x \in \Gc$, by Assumption (A2) - cf. \eqref{eq_P_dominated}
\begin{align*}
\abs{f(x)} \leq \int_{\Gc} \abs{f(y)} \kernelenv(y^{-1}x) dy
=
\int_{\Gc} \abs{f(y)} \iv{\kernelenv}(x^{-1}y) dy
\leq \norm{f}_2 \norm{\iv{\kernelenv}}_2, 
\end{align*}
which is finite because $\iv{\kernelenv} \in \winr \subseteq \Ltsp$.
Hence $f(x) = \ip{f}{E_x}$, for some function $E_x \in \SLt$
with $\norm{E_x}_2 \leq \norm{\iv{\kernelenv}}_2$.

Let $\sett{e_k}_k$ be an orthonormal basis of $\SLt$. Since by (a) $\multm$ is a positive operator,
it suffices to check that $\sum_k \ip{\multm e_k}{e_k} \lesssim \norm{m}_1$ (see for example,
\cite[Theorem 2.14]{si79}). To this end, note that for $x \in \Gc$,
\begin{align*}
\sum_k \abs{e_k(x)}^2 = \sum_k \abs{\ip{e_k}{E_x}}^2 = \norm{E_x}^2 \leq \norm{\iv{\kernelenv}}_2^2.
\end{align*}
Hence,
\begin{align*}
\sum_k \ip{\multm e_k}{e_k} &= \sum_k \int_\Gc m(x) \abs{e_k(x)}^2 \leq \norm{\iv{\kernelenv}}_2^2 \norm{m}_1.
\end{align*}
This completes the proof.
\end{proof}
\section{Well-spread families of operators}
\label{sec_abs_mol}
Throughout  this section, we assume (A1) and (A2) from Section~\ref{sec_model}. Recall the notion of phase-space
multiplier $M_m:\SEsp\to\SEsp, \, M_m(f) = P(mf)$, from Section \ref{sec_phmul}. If $P:\Esp \to \SEsp$ is the
projector from  (A2) and $m \in L^\infty(\Gc)$ is a symbol, then (A1), (A2) intuitively imply that a phase-space
multiplier spreads the mass of $f$ in a controlled way. Indeed, using the bound in (A2) we obtain 
\begin{align}
\label{eq_well_bound}
\abs{M_m f(x)} = \abs{P(mf)(x)}
\leq \int_\Gc \abs{m(y)} \abs{f(y)} K(y^{-1}x) dy, \qquad x \in \Gc.
\end{align}
Hence, if $m$ is known to be concentrated in a certain region of $\Gc$, so will be $M_m(f)$.

One of the main technical insights of this article is the fact that some important tools used in the investigation of 
families phase-space multipliers only depend on estimates
such as  \eqref{eq_well_bound}. To formalize this observation, we now introduce the
key concept of families of operators that are {\em well-spread} in phase-space,  implying that the operators are
dominated by product-convolution operators
 centered at suitably distributed
nodes $\gamma$.
\begin{definition}[Well-spread family of operators]\label{Def:wellspread}
Let $\Gamma \subseteq \Gc$ be a relatively separated set,  $\env \in W(L^\infty,L^1_w) \cap W_R(L^\infty,L^1_w)$
and $g \in \winr$ be non-negative functions. A family of operators
$\sett{\molg: \SEsp \to \SEsp: \gamma \in \Gamma}$ is
called \emph{well-spread with envelope} $(\Gamma, \env, g)$ if the following pointwise estimate holds
\begin{align}
\label{eq_deff_loc_mol}
\abs{\molg f(x)} \leq \int_{\Gc} g(\gamma^{-1}y) \abs{f(y)} \env(y^{-1}x) dy,
\qquad \gamma \in \Gamma, x \in \Gc.
\end{align}\end{definition}
If we do not want to emphasize the role of the envelope,  we say that
$\sett{\molg: \gamma \in \Gamma}$ is a well-spread family of operators, assuming the existence of an adequate
envelope.

The advantage of working with well-spread families instead of just families of phase-space multipliers
is that well-spreadness  is stable under various  operations, e.g. finite composition. This will be essential to the
proofs of our main results.

The canonical example of a well-spread family of operators is a family of phase-space multipliers
$\set{\multmg}{\gamma \in \Gamma}$ associated with an adequate family of symbols.
\begin{definition}[Well-spread family of symbols]\label{eq:etaupbound}
A family of (measurable) symbols $\set{\eta_\gamma:\Gc \to \bC}{\gamma \in \Gamma}$ is called \emph{well-spread} 
if 
\begin{itemize}
\item $\Gamma \subseteq \Gc$ is a relatively separated set and
\item there is a function $g \in \wright$ such that $
\abs{\eta_\gamma(x)} \leq g(\gamma^{-1}x)$, $x \in \Gc$, $\gamma \in \Gamma.
$
\end{itemize}
The pair $(\Gamma, g)$ is called an envelope for $\sett{\eta_\gamma: \gamma \in \Gamma}$.
\end{definition}
Together with the properties of the convolution kernel $K$ dominating the projection onto $\SEsp$
- cf. (A1), (A2) - we immediately obtain a family of well-spread operators.
\begin{prop}
\label{prop_covers_are_molecules}
Let $\sett{\eta_\gamma: \gamma \in \Gamma}$ be a well-spread family of symbols. Then the corresponding family of
phase-space multipliers
$\sett{\multmg: \gamma \in \Gamma}$ is well-spread.
\end{prop}
\begin{proof}
It follows readily from the definitions and Assumptions (A1), (A2) that if $(\Gamma, g)$ is an envelope for 
$\sett{\eta_\gamma: \gamma \in \Gamma}$, then $(\Gamma, \kernelenv, g)$ is an envelope for
$\sett{\multmg: \gamma \in \Gamma}$.
\end{proof}
The reason why we introduce the concept of well-spread families of operators is that composition of phase-space
multipliers usually fails to yield a phase-space multiplier. However, the estimate in \eqref{eq_deff_loc_mol}
is stable under various operations. In this article we will be mainly interested in finite composition and we have the following.
\begin{prop}
\label{prop_squared_covers_are_molecules}
Let $\sett{\eta_\gamma: \gamma \in \Gamma}$ be a well-spread family of symbols. Then the family of
operators $\sett{(\multmg)^2: \gamma \in \Gamma}$ is well-spread.
\end{prop}
\begin{proof}
If $(\Gamma, g)$ is an envelope for $\sett{\eta_\gamma: \gamma \in \Gamma}$ then
\begin{align*}
\abs{(\multmg)^2 f(x)} &\leq \int_\Gc g(\gamma^{-1} y) \abs{\multmg f(y)} \kernelenv(y^{-1}x) dy
\\
&\leq
\int_\Gc \int_\Gc g(\gamma^{-1} y) g(\gamma^{-1} z) \abs{f(z)} \kernelenv(z^{-1}y) \kernelenv(y^{-1}x) dy dz
\\
&\leq \norm{g}_\infty
\int_\Gc  g(\gamma^{-1} z) \abs{f(z)} (\kernelenv* \kernelenv )(z^{-1}x) dz.
\end{align*}
Hence, if we set $\env := \kernelenv*\kernelenv$, it follows that $(\Gamma, \env, \norm{g}_\infty g)$ is an
envelope for $\sett{(\multmg)^2: \gamma \in \Gamma}$. Since $\kernelenv$ belongs to $W(L^\infty,L^1_w)
\cap W_R(L^\infty,L^1_w)$, by Proposition \ref{prop_conv_algebra}, so does $\env$.
\end{proof}
\subsection{Almost-orthogonality estimates}
We now introduce one of the key estimates used in this article,
that can be seen as a generalization of \eqref{eq_intro_norm_equiv}. 
\begin{theorem}
\label{th_almost_orth}
Let $\sett{\molg: \gamma \in \Gamma}$ be a well-spread family of operators. Suppose that the operator
$\sum_\gamma \molg: \SEsp \to \SEsp$ is invertible. Then the following norm equivalence holds
\begin{align}
\label{eq_almost_orth}
\norm{f}_\Esp
\approx
\norm{( \norm{\molg f}_{L^2(\Gc)} )_{\gamma \in \Gamma}}_{\Espd}, \qquad f \in \SEsp.
\end{align}
\end{theorem}
\begin{rem}
The implicit constants depend only on $\norm{\kernelenv}_{W(L^\infty,L^1_w)}$,
$\norm{\kernelenv}_{W_R(L^\infty,L^1_w)}$,
$\norm{\env}_{W(L^\infty,L^1_w)}$,
$\norm{\env}_{W_R(L^\infty,L^1_w)}$,
$\norm{g}_{\wright}$,
$\rel(\Gamma)$ (cf. \eqref{eq_spread})
and $\consEw$ (cf. \eqref{weight_w_admissible}),
and the norms of $\sum_\gamma \molg$ and $(\sum_\gamma \molg)^{-1}$.
\end{rem}
Theorem~\ref{th_almost_orth}, in its version for families of phase-space multipliers associated with
well-spread families of symbols, is proved in \cite{ro12}.  However, the arguments involved work unaltered for the
case of general well-spread families of operators. Indeed, Definition \ref{Def:wellspread} is tailored to the
requirements of the proof in \cite{ro12}. For completeness, we sketch a proof of Theorem \ref{th_almost_orth} in
Appendix A.

Theorem \ref{th_almost_orth} is an almost orthogonality principle: because of the phase-space localization of the
family $\sett{\molg: \gamma \in \Gamma}$, the effects of each individual operator within the sum $\sum_\gamma \molg$
decouple.

We also point out that the $L^2$-norm in \eqref{eq_almost_orth} can be replaced by any solid
translation invariant norm (see \cite{ro12}). However, the by far most important case is the one of
$L^2$. Indeed, in this article we exploit \eqref{eq_almost_orth} to extrapolate certain 
thresholding estimates from $L^2$ to other Banach spaces, cf.~Theorem~\ref{th_almost_orth_tf}.

As a consequence of Theorem~\ref{th_almost_orth} we get the following fact.
\begin{coro}
\label{coro_squares}
Let $\sett{\molg: \gamma \in \Gamma}$ be a well-spread family of operators. Suppose that $\Esp=L^2(\Gc)$ and that
the operator $\sum_\gamma \molg: \SLt \to \SLt$ is invertible. Then so is the operator $\sum_\gamma \admolg \molg: \SLt
\to
\SLt$.
\end{coro}
\begin{proof}
By Theorem \ref{th_almost_orth}, the invertibility of $\sum_\gamma \molg$ implies that for $f \in \SLt$,
\begin{align*}
\norm{f}_2^2 \approx
\sum_{\gamma\in\Gamma} \norm{\molg f}_2 ^2 = \ip{\sum_{\gamma\in\Gamma} \admolg \molg f}{f}.
\end{align*}
Hence $\sum_\gamma \admolg \molg$ is positive definite and therefore invertible on $\SLt$.
\end{proof}
\begin{rem}
Consider a well-spread family of self-adjoint operators
$\sett{\molg: \gamma \in \Gamma}$. Since
$(\sum_\gamma \molg)^2=\sum_\gamma (\molg)^2 + \sum_{\gamma, \gamma' / \gamma \not= \gamma'} \molg \molgp$
and $(\sum_\gamma \molg)^2$ is invertible if and only if $\sum_\gamma \molg$ is, Corollary \ref{coro_squares}
says that the invertibility of $(\sum_\gamma \molg)^2$ implies that of its ``diagonal part'' $\sum_\gamma (\molg)^2$.
This gives further support to the interpretation of Theorem \ref{th_almost_orth} as an almost orthogonality principle.
\end{rem}

\section{The case of Time-Frequency analysis}
\label{sec_tf}
In this section we show how time-frequency analysis fits the abstract model from Section \ref{sec_abstract}
and also illustrate the notions from Section \ref{sec_abs_mol} in this particular case. We first introduce the 
relevant class of function spaces.

\subsection{Modulation spaces}
\label{sec_mod_sp}
Let $\varphi \in \mathcal{S}(\Rdst)$ be a non-zero function normalized by $\norm{\varphi}_2=1$ and recall that the
STFT of a distribution $f \in \mathcal{S}'(\Rdst)$ is the function
$V_\varphi f: \Rdst \times \Rdst \to \bC$ given by
\begin{align*}
\Vstft_{\varphi}f(z)
=\int_{\mathbb{R}^d}f(t)\overline{\varphi(t-x)}e^{-2\pi i\xi t}dt,
\quad z=(x,\xi) \in \Rdst \times \Rdst.
\end{align*}
Modulation spaces are defined by imposing weighted $L^p$-norms on the STFT, as we now describe.
As weight we consider continuous functions $v: \Rtdst \to (0,+\infty)$ such that $v(z+z') \leq C v(z)(1+\abs{z'})^N$,
for some constants $N,C>0$ and all $z,z' \in \Rtdst$. Such functions are called \emph{polynomially moderated weights}.
For a polynomially moderated weight $v$ and $p \in [1,+\infty]$, the \emph{modulation space} $M^p_v(\Rdst)$ is
\begin{align}
M^p_v(\Rdst) := \set{f \in \mathcal{S}'(\Rdst)}{V_\varphi f \in L^p_v(\Rtdst)}.
\end{align}
The space $M^p_v(\Rdst)$ is given the norm
\begin{align}
\norm{f}_{M^p_v(\Rdst)} := \norm{V_\varphi f}_{L^p_v(\Rtdst)}=
\left(\int_{\Rtdst} \abs{f(z)}^p v(z)^p dz\right)^{1/p},
\end{align}
with the usual modification for $p=+\infty$. A different choice of $\varphi$ yields the same space
with an equivalent norm, $M^p_v(\Rdst)$ is a Banach space,
and $V_\varphi: M^p_v(\Rdst) \to V_\varphi(M^p_v(\Rdst))$ defines an isometric isomorphism,
where $V_\varphi(M^p_v(\Rdst))$ is considered as a subspace of $L^p_v(\Rtdst)$
(see ~\cite[Proposition 4.3]{fegr89} or \cite[Chapter 11]{gr01}).
The unweighted modulation space $M^2(\Rdst)$ coincides with $L^2(\Rdst)$.

The assumption that $\varphi$ is a Schwartz function can be relaxed by considering 
the following classes of weights.
\begin{definition}[Admissible TF weights]
\label{def_tf_weights}
A continuous even function $w: \Rtdst \to (0,+\infty)$ is called an \emph{admissible TF weight} if it satisfies the
following.
\begin{itemize}
\item[(i)] $w$ is submultiplicative, i.e. $w(z+z') \leq w(z)w(z')$, for all $z,z' \in \Rtdst$.
\item[(ii)] $w(z) \leq C(1+\abs{z})^N$, for some constants $C,N>0$ and all $z  \in \Rtdst$.
\item[(iii)] $w$ satisfies the following condition - known as the GRS-condition after Gelfand, Raikov and Shilov
\cite{gerash64}
\begin{align*}
\lim_{n \rightarrow \infty} w(nz)^{1/n}=1,
\quad z \in \Rtdst.
\end{align*}
\end{itemize}
\end{definition}
The polynomial weights $w(z) = (1+\abs{z})^s$ with $s \geq 0$ are
examples of admissible TF weights. 
A second weight $v: \Rtdst \to (0,+\infty)$ is said to be $w$-moderated if
\begin{align}
\label{eq_v_moderated}
v(z+z') \leq C_v w(z) v(z'),
\quad z,z' \in \Rtdst,
\end{align}
for some constant $C_v$. Note that in this case $v$ is polynomially moderated because $w$ is dominated by a polynomial.
(Do not confuse the notion of admissible TF weight with the one of being a weight
\emph{admissible for a certain BF space} - cf. Definition \ref{def_adm_for}.)

Let $w$ be an admissible TF weight and let $\varphi \in M^1_w(\Rdst)$ be non-zero. For every $w$-moderated weight $v$
and $p \in [1,+\infty]$ the modulation space $M^p_v(\Rdst)$ can be described as
\begin{align}
M^p_v(\Rdst) = \set{f \in M^\infty_{1/w}(\Rdst)}{V_\varphi f \in L^p_v(\Rtdst)},
\end{align}
and $\norm{V_\varphi f}_{L^p_v(\Rtdst)}$ is an equivalent norm on it (see \cite[Chapter 11]{gr01}). Hence,
the assumption: $\varphi \in M^1_w(\Rdst)$ is enough to treat the class of all modulation spaces
$M^p_v$, with $v$ a $w$-moderated weight and $p \in [1,+\infty]$. (The GRS-condition above does not play a role in what
was presented so far; it is required for the technical results used in Sections \ref{sec_wiener} and
\ref{sec_tf_spread}.)

\subsection{Phase-space and modulation spaces}
\label{sec_fs_and_mod}
Extending Example \ref{ex_tf}, in this section we explain how modulation spaces fit the model of Section
\ref{sec_abstract}.

Let $w$ be an admissible TF weight, $\varphi \in M^1_w(\Rdst)$ normalized by $\norm{\varphi}_2=1$, and $v$
a $w$-moderated weight. To apply the model we let $\mathcal{G} = \mathbb{R}^{2d}$ with the usual group structure and
$\Esp := L^p_v(\Rtdst)$. Since $v$ is $w$-moderated, we deduce readily that $w$ is admissible for $\Esp$.

We also let $\SEsp := V_\varphi M^p_v(\Rdst)$. The fact that $\SEsp$ is closed within $L^p_v(\Rtdst)$ follows from the
fact that $V_\varphi: M^p_v(\Rdst) \to \SEsp$ is an isometric isomorphism
and $M^p_v(\Rdst)$ is complete. As projection $P: L^p_w(\Rtdst) \to V_\varphi
M^p_v(\Rdst)$ we consider the operator composition of $V_\varphi$ with its formal adjoint: $P=V_\varphi V_\varphi^*$. As
discussed in Section \ref{sec_bla_1} and Example \ref{ex_tf}, this operator is well-defined and satisfies
\begin{align*}
\abs{P F(x',\xi')} \leq
\int_{\Rdst \times \Rdst} \abs{F(x,\xi)} \abs{V_\varphi \varphi(x'-x,\xi'-\xi)} dx d\xi,
\qquad F \in L^p_v(\Rtdst).
\end{align*}
The fact that $\varphi \in M^1_w(\Rdst)$ in principle means that $V_\varphi \varphi \in L^1_w(\Rtdst)$,
but it also follows that $V_\varphi \varphi \in \wrtd$ - see for example \cite[Proposition 12.1.11]{gr01}.

Hence, (A2)-(A3) are satisfied with $K:=\abs{V_\varphi \varphi}$. (Note that since $\Gc=\Rtdst$ is
abelian, the left and right amalgam spaces coincide.)

Recall the time-frequency localization operator with symbol $m \in L^\infty(\Rtdst)$
formally given by $H_m f = V^*_\varphi (m V_g f)$, cf.~\eqref{eq_loc_op_2}.
In Section \ref{sec_bla_1} we discussed the operator $H_m: \LtRd \to \LtRd$. 
Analogous considerations apply to $H_m: M^p_v(\Rdst) \to M^p_v(\Rdst)$ and show that
$V_\varphi H_m V^*_\varphi F = P ( m F)$ for all $ F \in V_\varphi(M^p_v(\Rdst))$. 
This means that under the isometric isomorphism $V_\varphi: M^p_v(\Rdst) \to V_\varphi(M^p_v(\Rdst))$,
the time-frequency localization operator $H_m$ corresponds to phase-space multiplier $M_m$ from Section \ref{sec_phmul}.
\begin{rem}
In this article we will be concerned with time-frequency localization operators with bounded symbols $m$.
However the condition that $m$ be bounded is by no means necessary for $H_m:M^p_v(\Rdst) \to M^p_v(\Rdst)$ to
be bounded. See for example \cite{feno03,cogr03} for sharper boundedness results for time-frequency localization
operators.
\end{rem}

\subsection{Spectral invariance}
\label{sec_wiener}
We present a technical result that says that for a certain class of operators, the property of being invertible
on $L^2(\Rdst)$ automatically implies invertibility on a range of modulation spaces. The following proposition
is obtained by combining known spectral invariance results for pseudodifferential
operators \cite{sj94,sj95,gr06,gr06-3}. For convenience of the reader, in Appendix B we provide
the necessary background results and give a proof based on those results.
\begin{prop}
\label{prop_wiener}
Let $w:\Rtdst \to (0,+\infty)$ be an admissible TF weight - cf. Section \ref{sec_mod_sp} -
and let $\varphi \in M^1_w(\Rdst)$ be non-zero.
Let $v$ be a $w$-moderated weight and let $p \in [1,+\infty]$. Let $T:M^p_v(\Rdst) \to M^p_v(\Rdst)$
be an operator satisfying the enveloping condition
\begin{align}
\label{eq_wiener_env}
\abs{V_\varphi T(f)(z)} \leq \int_{\Rtdst} \abs{V_\varphi f(z')} H(z-z') dz',
\qquad z \in \Rtdst,
\end{align}
for some function $H \in L^1_w(\Rtdst)$. Assume that $T:L^2(\Rdst) \to L^2(\Rdst)$ is invertible.
Then $T: M^p_v(\Rdst) \to M^p_v(\Rdst)$ is invertible.
\end{prop}

\subsection{Families of operators well-spread in the time-frequency plane}
\label{sec_tf_spread}
We now illustrate the notions and results from Section \ref{sec_abs_mol} in the case of time-frequency analysis. 

Let $w:\Rtdst \to (0,+\infty)$ be an admissible TF weight and $\varphi \in M^1_w(\Rdst)$ with
$\norm{\varphi}_2=1$. Setting $\mathcal{G} = \Rtdst$ in~\eqref{eq_spread}, we call a set
$\Gamma \subseteq \Rtdst$ relatively separated if
$\sup_{z \in \Rtdst} \#(\Gamma \cap B_1(z)) < +\infty$.
In analogy to Definition~\ref{eq:etaupbound}, a family of symbols $\set{\eta_\gamma:\Rtdst \to \bC}{\gamma \in \Gamma}$
is \emph{well-spread} (relative to $w$) if $\Gamma$ is relatively separated and there exists $g \in \wrtd$ such that
\begin{align*}
\abs{\eta_\gamma(z)} \leq g(z-\gamma), \qquad z \in \Rtdst, \gamma \in \Gamma.
\end{align*}
Similarly, in analogy to Definition~\ref{Def:wellspread}, a family of operators $\set{\tfmolg: \Miwp(\Rdst) \to
\Miwp(\Rdst)}{\gamma \in \Gamma}$ is said to be \emph{well-spread in the time-frequency plane} (relative to
$(\varphi,w)$) if there exists an envelope $(\Gamma, \env, g)$, with $\Gamma \subseteq \Rtdst$ a
relatively separated set and $\env, g \in \wrtd$ such that
\begin{align}
\label{eq_well_spread_tf}
\abs{\stft \tfmolg f(z)} \leq \int_{\Rtdst} g(z'-\gamma) \abs{\stft f(z')} \env(z-z') dz',
\qquad z \in \Rtdst, \gamma \in \Gamma.
\end{align}
Note that $\sett{\tfmolg: \gamma \in \Gamma}$ is well-spread in the TF plane if and only if
the family of operators $\set{\stft \tfmolg \stft^*: 
V_\varphi(M^\infty_{1/w}(\Rdst)) \to V_\varphi (M^\infty_{1/w}(\Rdst))}{\gamma \in \Gamma}$
is well-spread in the sense of Section \ref{sec_abs_mol}.

Note also that, if $\sett{\tfmolg: \gamma \in \Gamma}$ is well-spread in the time-frequency plane, then,
because of \eqref{eq_well_spread_tf}, for all $w$-moderated weights $v$ and $p \in [1,+\infty]$
each operator $\tfmolg$ maps $\Mpv(\Rdst)$ into $\Mpv(\Rdst)$ with a norm bound independent
of $\gamma$,
\begin{align*}
\norm{\tfmolg f}_{\Mpv} \leq C_v \norm{g}_{\infty} \norm{\env}_{L^1_w} \norm{f}_{\Mpv}.
\end{align*}
\begin{rem}
In parallel to Proposition~\ref{prop_covers_are_molecules}, if $\sett{\eta_\gamma: \gamma \in \Gamma}$ is a well-spread
family of symbols, then the corresponding family of time-frequency localization operators
$\sett{H_{\eta_\gamma}:\gamma \in \Gamma}$ is well-spread in the time-frequency plane.
\end{rem}

In the case of time-frequency analysis we can strengthen Theorem~\ref{th_almost_orth} by means of
the spectral invariance result in Proposition  \ref{prop_wiener}. Due to this result, the
invertibility assumption in Theorem~\ref{th_almost_orth} can be replaced by assuming invertibility on $\LtRd$.

\begin{theorem}
\label{th_almost_orth_tf}
Let $w:\Rtdst \to (0,+\infty)$ be an admissible TF weight and let $\varphi \in M^1_w(\Rdst)$.
Let $\sett{\tfmolg: \gamma \in \Gamma}$ be well-spread in the TF plane (relative to $(\varphi,w)$).

Suppose that the operator $\sum_\gamma \tfmolg: \LtRd \to \LtRd$ is invertible. Then, 
for all $w$-moderated weights $v$ and for all $1 \leq p \leq +\infty$,
$\sum_\gamma \tfmolg: \Mpv \to \Mpv$ is invertible and the following norm equivalence holds
\begin{align*}
\norm{f}_\Mpv
\approx
\left(\sum_{\gamma \in \Gamma} \norm{\tfmolg f}_{L^2(\Rdst)}^p v(\gamma)^p \right)^{1/p},
\qquad f \in \Mpv(\Rdst),
\end{align*}
with the usual modification for $p=\infty$.
\end{theorem}
\begin{rem}
The estimates hold uniformly for $1 \leq p \leq +\infty$ and any family of weights
having a uniform constant $C_v$ (cf. \eqref{eq_v_moderated}).
\end{rem}
\begin{proof}[Proof of Theorem \ref{th_almost_orth_tf}]
Let $v$ be a $w$-moderated weight and $1 \leq p \leq +\infty$. 
By hypothesis $\sum_\gamma \tfmolg: L^2(\Rdst) \to L^2(\Rdst)$ is invertible. In order to apply Proposition
\ref{prop_wiener} to this operator we let $f \in M^p_v(\Rdst)$ and use \eqref{eq_well_spread_tf} to estimate
\begin{align*}
\abs{V_\varphi\Big(\sum_{\gamma \in \Gamma} \tfmolg f\Big)(z)}
&\leq
\int_{\Rtdst} \abs{\stft f(z')} \env(z-z') \sum_{\gamma \in \Gamma} g(z'-\gamma) dz'
\\
&\lesssim \norm{g}_\win \int_{\Rtdst} \abs{\stft f(z')} \env(z-z') dz'.
\end{align*}
Since $\env \in \win(\Rtdst)$, it follows from Proposition
\ref{prop_wiener} that $\sum_\gamma \tfmolg: M^p_v(\Rdst) \to M^p_v(\Rdst)$ is invertible.

We use the notation $\Spv := V_\varphi(M^p_v(\Rdst))$.
Let $\molg := \stft \tfmolg \stft^*: \Spv \to \Spv$. The family $\sett{\molg: \gamma \in \Gamma}$
is well-spread (in the sense of Section \ref{sec_abs_mol}). Since
$\sum_\gamma \tfmolg: M^p_v(\Rdst) \to M^p_v(\Rdst)$ is invertible, so is
$\sum_\gamma \molg: \Spv \to \Spv$. Therefore, we can apply 
Theorem \ref{th_almost_orth} and the conclusion follows.
\end{proof}

\section{Frames of eigenfunctions of time-frequency localization operators}
\label{sec_frames_tf}
In this section we prove the main results on constructing frames with a prescribed time-frequency profile.
Throughout this section let $w:\Rtdst \to (0,+\infty)$ be an admissible TF weight and $\varphi \in M^1_w(\Rdst)$ be
normalized by $\norm{\varphi}_2=1$. 
\subsection{Thresholding eigenvalues}
\label{sec_thresh}
Let $\sett{\eta_\gamma:\gamma\in\Gamma}$ be a family of non-negative functions on $\Rtdst$ 
that is well-spread (relative to $w$) and consider the corresponding family of time-frequency localization operators 
$
\locmg f = \stft^*(\eta_\gamma \stft f)$.
Since each $\eta_\gamma$ is non-negative and belongs to $\Lisp(\Rtdst)$, $\locmg$ is positive and trace class, 
and $\trace(\locmg)=\norm{\eta_\gamma}_1$, see
\cite{boco02,feno03,wo02, boto08} and Proposition~\ref{prop_mult_hilbert}.
Hence $\locmg: \LtRd \to \LtRd$ can be diagonalized as
\begin{align}
\label{eq_eigen_locmg}
\locmg f = \sum_{k \geq 1} \lambdagk \ip{f}{\phigk}{\phigk},
\qquad f \in \LtRd
\end{align}
where $\set{\phigk}{k \in \NN}$ is an orthonormal subset of $\LtRd$ 
- possibly incomplete if $\ker(\locmg) \not= \sett{0}$ - and $(\lambdagk)_k$
is a non-increasing sequence of non-negative real numbers.

The time-frequency profile of the functions $\set{\phigk}{k \in \NN}$ is optimally adapted to the mask $\eta_\gamma$
in the following sense. For each $N \in \Nst$, the set $\sett{\phi^\gamma_1, \ldots, \phi^\gamma_N}$ is an orthonormal
set maximizing the quantity,
\begin{align}
\label{eq_optimal}
\sum_{k=1}^N \int_\Rtdst \eta_\gamma(z) \abs{\stft f_k(z)}^2 \,dz,
\end{align}
among all orthonormal sets $\sett{f_1, \ldots, f_N} \subseteq \LtRd$. 

Moreover, since $\eta_\gamma \in L^1_w$ and $\varphi \in M^1_w$, $\phigk \in M^1_w$
provided that $\lambdagk \not= 0$ (see for example \cite[Lemma 5]{dogr11}).

For every $\varepsilon >0$, we define the operator $\locmgeps$ by applying a threshold to the eigenvalues of $\locmg$,
\begin{align}
\label{eq_deff_locgeps}
\locmgeps f := \sum_{k:\lambdagk>\varepsilon} \lambdagk \ip{f}{\phigk}{\phigk},
\qquad f \in \LtRd.
\end{align}
Hence,
\begin{align}
\label{eq_Lg_Lgeps}
\norm{\locmgeps f}_2 \leq \norm{\locmg f}_2 \leq \norm{\locmgeps f}_2 + \varepsilon \norm{f}_2,
\qquad f \in \LtRd.
\end{align}

As a first step to analyze the effect of the thresholding operation $\locmg \mapsto \locmgeps$,
we show that  the operators $\sett{\locmg: \gamma \in \Gamma}$ behave ``globally'' like projectors. 
Since these operators may have infinite rank, in general 
$\norm{(\locmg)^2 f}_2 \not\approx \norm{\locmg f}_2$. 
However, we prove the following.

\begin{prop}
\label{prop_squared_covers}
Let $w:\Rtdst \to (0,+\infty)$ be an admissible TF weight and let $\varphi \in M^1_w(\Rdst)$.
Let $\sett{\eta_\gamma: \gamma\in\Gamma}$ be a well-spread 
family of non-negative symbols on $\Rtdst$ (relative to $w$) with $\sum_\gamma
\eta_\gamma \approx 1$. Then for all $w$-moderated weight $v$ and all $p \in [1,+\infty]$
\begin{align}
\norm{f}_\Mpv
\approx
\left(\sum_{\gamma \in \Gamma} \norm{\locmg f}_{L^2(\Rdst)}^p v(\gamma)^p \right)^{1/p}
\approx
\left(\sum_{\gamma \in \Gamma} \norm{(\locmg)^2 f}_{L^2(\Rdst)}^p v(\gamma)^p \right)^{1/p},
\qquad
f \in \Mpv(\Rdst).
\end{align}
\end{prop}
\begin{rem}
The estimate $\norm{f}_\Mpv \approx \norm{ (\norm{\locmg f}_2)_\gamma }_\ellpv$
is contained in \cite{dofegr06, dogr11} for the case of families of symbols consisting of lattice translates of a single
function, and in \cite{ro12} in the present generality. The norm equivalence involving $(\locmg)^2 f$ is new.
\end{rem}
\begin{proof}[Proof of Proposition \ref{prop_squared_covers}]
Since the symbols $\eta_\gamma$ satisfy $m:=\sum_\gamma \eta_\gamma \geq A$ for some constant $A>0$ it follows that
\begin{align*}
&\ip{\sum_{\gamma\in\Gamma} \locmg f}{f} = \ip{\locm f}{f} = \ip{\stft^*(m \stft f)}{f}
\\ &\qquad= \ip{m\stft f}{\stft f} = \int_\Rtdst m(z) \abs{\stft f(z)}^2 dz
\geq A \norm{\stft f}^2_2
= A \norm{f}^2_2.
\end{align*}
Hence, $\sum_\gamma \locmg: \LtRd \to \LtRd$ is invertible. 
Now Proposition \ref{prop_covers_are_molecules} and Theorem \ref{th_almost_orth_tf} yield the estimate
involving $\norm{f}_\Mpv$ and $\norm{ (\norm{\locmg f}_2)_\gamma }_\ellpv$.

For the estimate involving the squared operators, note that Corollary \ref{coro_squares} implies that
$\sum_\gamma (\locmg)^2: \SLt \to \SLt$ is also invertible. According to Proposition
\ref{prop_squared_covers_are_molecules},
the family $\sett{(\locmg)^2: \gamma \in \Gamma}$ is also well-spread. Hence, a second application
of Theorem \ref{th_almost_orth_tf} concludes the proof.
\end{proof}

We now state and prove that a sufficiently fine thresholding of the eigenvalues
of the operators $\sett{\locmg: \gamma \in \Gamma}$ still decomposes the family of modulation spaces.
\begin{theorem}
\label{th_finite_rank_tf}
Let $w:\Rtdst \to (0,+\infty)$ be an admissible TF weight and let $\varphi \in M^1_w(\Rdst)$.
Let $\sett{\eta_\gamma: \gamma \in \Gamma}$ be a family of non-negative symbols on $\Rtdst$ 
that is well-spread (relative to $w$) and such that $\sum_\gamma \eta_\gamma
\approx 1$. Let $v$ be a $w$-moderated weight. Then there exist constants $0<c \leq C < +\infty$
such that for all sufficiently small $\varepsilon >0$ and all $p \in [1,+\infty]$
\begin{align*}
c \norm{f}_\Mpv \leq
\left(\sum_{\gamma \in \Gamma} \norm{\locmgeps f}_{L^2(\Rdst)}^p v(\gamma)^p \right)^{1/p}
\leq C \norm{f}_\Mpv,
\qquad f \in \Mpv(\Rdst),
\end{align*}
with the usual modification for $p=\infty$.

The choice of $\varepsilon$ and the estimates are uniform for $1 \leq p \leq +\infty$ and any family of weights
having a uniform constant $C_v$ (cf. \eqref{eq_v_moderated}).
\end{theorem}
\begin{proof}
Given  $\varepsilon>0$ and $f \in \Mpv(\Rdst)$ we apply \eqref{eq_Lg_Lgeps} to $\locmg(f)$ to obtain
\begin{align*}
\norm{(\locmg)^2 f}_2 \leq \norm{\locmgeps \locmg f}_2 + \varepsilon \norm{\locmg f}_2.
\end{align*}
Since $\locmgeps$ and $\locmg$ commute,
$\norm{\locmgeps \locmg f}_2 = \norm{\locmg \locmgeps f}_2
\lesssim \norm{\locmgeps f}_2$. Hence,
\begin{align*}
\norm{(\locmg)^2 f}_2 \lesssim \norm{\locmgeps f}_2 + \varepsilon \norm{\locmg f}_2.
\end{align*}
Taking $\ellpv$ norm on $\gamma$ yields
\begin{align*}
\norm{( \norm{(\locmg)^2 f}_2 )_\gamma}_\ellpv
\lesssim
\norm{( \norm{\locmgeps f}_2)_\gamma)}_\ellpv
+ \varepsilon \norm{( \norm{\locmg f}_2 )_\gamma}_\ellpv.
\end{align*}
Using the estimates in Proposition \ref{prop_squared_covers} we get
\begin{align*}
\norm{f}_\Mpv \leq C \norm{ (\norm{\locmgeps f}_2)_\gamma }_\ellpv
+ c \varepsilon \norm{f}_\Mpv,
\end{align*}
for some constants $c,C$. Hence,
\begin{align*}
(1-c\varepsilon) \norm{f}_\Mpv \leq C \norm{ (\norm{\locmgeps f}_2)_\gamma }_\ellpv.
\end{align*}
This gives the desired lower bound (for all $0<\varepsilon <1/c$).
The upper bound follows from the first inequality in
\eqref{eq_Lg_Lgeps} and Proposition \ref{prop_squared_covers}.
\end{proof}

\begin{rem}
Note that the proof of Theorem \ref{th_finite_rank_tf} only uses the estimate in \eqref{eq_Lg_Lgeps} and the fact
that $\locmg$ and $\locmgeps$ commute. Hence, more general thresholding rules besides
the one in \eqref{eq_deff_locgeps} can be used.
\end{rem}

\subsection{Frames of eigenfunctions}\label{Sec:FoE}
Finally, we obtain the desired result on frames of eigenfunctions.

\begin{theorem}
\label{th_frames_tf}
Let $w:\Rtdst \to (0,+\infty)$ be an admissible TF weight and let $\varphi \in M^1_w(\Rdst)$.
Let $\sett{\eta_\gamma: \gamma\in\Gamma}$ be a family of non-negative symbols on $\Rtdst$ 
that is well-spread (relative to $w$) and such that $\sum_\gamma \eta_\gamma \approx 1$.
Let $v$ be a $w$-moderated weight. Then there exists a constant $\alpha>0$
such that, for every choice of finite subsets of eigenfunctions of $\locmg$
$\set{\phigk}{\gamma \in \Gamma, 1 \leq k \leq N_\gamma}$ with 
\begin{align*}
\alpha \norm{\eta_\gamma}_1 \leq N_\gamma \mbox{ and }
\sup_{\gamma \in \Gamma} N_\gamma < +\infty,
\end{align*}
the following frame estimates hold simultaneously for all $1 \leq p \leq +\infty$, with the usual modification for $p=
\infty$:
\begin{align*}
\norm{f}_{\Mpv} \approx
\left(
\sum_{\gamma \in \Gamma} \sum_{k=1}^{N_\gamma} \abs{\ip{f}{ \lambdagk \phi^{\gamma}_k}}^p v(\gamma)^p
\right)^{1/p},
\qquad f \in M^p(\Rdst).
\end{align*}
Moreover, $\alpha$ can be chosen uniformly for any class of weights $v$ having a uniform constant $C_v$
(cf. \eqref{eq_v_moderated}).
\end{theorem}
Before proving Theorem \ref{th_frames_tf} we make some remarks.
\begin{rem}
As opposed to other constructions that partition either the time or frequency domain (see e.g.
\cite{fegr85, alcamo04-1, boho06, camoro11}), the symbols $\eta_\gamma$ partition the time-frequency plane
simultaneously.
\end{rem}
\begin{rem}
Note that, if $(\Gamma, \env, g)$ is an envelope for $\sett{\eta_\gamma: \gamma \in \Gamma}$,
since $\norm{\eta_\gamma}_1 \leq \norm{g}_1$, in Theorem~\ref{th_frames_tf}
it is always possible to make a uniform choice $N_\gamma = N_0$.
\end{rem}
\begin{proof}[Proof of Theorem~\ref{th_frames_tf}]
For every $\gamma \in \Gamma$ and $\varepsilon>0$, let
$I^\varepsilon_\gamma := \sett{k \in \NN / \lambdagk > \varepsilon}$, which is a finite set.
Using Theorem \ref{th_finite_rank_tf}, Proposition~\ref{prop_squared_covers} 
and the orthonormality of the eigenfunctions, we can find a value of $\varepsilon>0$ such that
\begin{align*}
\norm{f}_{\Mpv} &\approx
\Big(
\sum_{\gamma \in \Gamma} \big(\sum_{k\in \NN} \abs{\ip{f}{\lambdagk\phigk}}^2 \big)^{p/2} v(\gamma)^p
\Big)^{1/p}
\\
&\approx\Big(
\sum_{\gamma\in\Gamma} \big(\sum_{k \in I^\varepsilon_\gamma} \abs{\ip{f}{\lambdagk\phigk}}^2 \big)^{p/2}
v(\gamma)^p
\Big)^{1/p},\end{align*}
with the usual modification for $p=+\infty$. This implies that for any choice of subsets
of indices $J^\varepsilon_\gamma \supseteq I^\varepsilon_\gamma$ we also have,
\begin{align}
\label{eq_Jgamma}
\norm{f}_{\Mpv} \approx
\Big(
\sum_{\gamma\in\Gamma} \big(\sum_{k \in J^\varepsilon_\gamma} \abs{\ip{f}{\lambdagk\phigk}}^2 \big)^{p/2} v(\gamma)^p
\Big)^{1/p}.
\end{align}
Furthermore, since $\sum_k \lambdagk = \trace(\locmg)= \norm{\eta_\gamma}_1$, we have 
$\# I^\varepsilon_\gamma \leq \varepsilon^{-1} \sum_k \lambdagk = \varepsilon^{-1} \norm{\eta_\gamma}_1$.
Hence, setting $\alpha := \varepsilon^{-1}$, we ensure that
for $N_\gamma \geq \alpha \norm{\eta_\gamma}_1$, the set
$J^\varepsilon_\gamma := \set{k \in \NN}{1 \leq k \leq N_\gamma}$ contains $I^\varepsilon_\gamma$ and therefore
satisfies \eqref{eq_Jgamma}.
Let $N:= \sup_\gamma N_\gamma$. By hypothesis, $N<\infty$.
Since $\# J^\varepsilon_\gamma = N_\gamma \leq N$, it follows that
\begin{align*}
(\sum_{k \in J^\varepsilon_\gamma} \abs{\ip{f}{\lambdagk\phigk}}^2 )^{1/2}
\approx
(\sum_{k \in J^\varepsilon_\gamma}\abs{\ip{f}{\lambdagk\phigk}}^p )^{1/p},
\end{align*}
with a constant that depends on $N$ (and the usual modification for $p=+\infty$). Combining the last estimate
with \eqref{eq_Jgamma} we obtain the desired conclusion.
\end{proof}
As an application of Theorem~\ref{th_frames_tf} we can now prove Theorem \ref{th_sample} 
(stated in the Introduction).
\begin{proof}[Proof of Theorem \ref{th_sample}]
Let $\sett{\Omega_\gamma: \gamma \in \Gamma}$ be an admissible cover of $\Rtdst$,
as defined in the Introduction - cf. \eqref{eq_admissibility_no_inner}. Let $\eta_\gamma := 1_{\Omega_\gamma}$ be the
characteristic function of $\Omega_\gamma$. The fact that $\sett{\Omega_\gamma: \gamma \in \Gamma}$ covers $\Rtdst$
implies that $\sum_\gamma \eta_\gamma = \sum_\gamma 1_{\Omega_\gamma} \geq 1$. In addition, 
if we let $g := 1_{B_R(0)}$, it follows from \eqref{eq_admissibility_no_inner} that $\sett{\eta_\gamma: \gamma \in
\Gamma}$ is well-spread with envelope $(\Gamma,g)$ and constant weight $w \equiv 1$. We note that $\abs{\Omega_\gamma} =
\norm{\eta_\gamma}_1$ and apply Theorem \ref{th_frames_tf} with $p=2$ and $v \equiv 1$. (Recall that
$M^2(\Rdst)=L^2(\Rdst)$.)
\end{proof}
\subsection{Inner regularity}
Finally we derive a variant of Theorem \ref{th_frames_tf} where, under an inner regularity assumption 
on the family of symbols, we renormalize the frame of eigenfunctions so that each frame element has norm 1.
To this end we first prove the following lemma which may be of independent interest.

\begin{lemma}
\label{lemma_loc_infinite_rank}
Let $\varphi \in \LtRd  \backslash\sett{0}$ and let $\Omega \subseteq \Rdst$ 
be a measurable set with non-empty interior. Then the time-frequency
localization operator $H_\Omega: \LtRd \to \LtRd$,
\begin{align*}
H_\Omega f(t) = \int_{\Omega} \Vstft_{\varphi}f(x,\xi) \varphi(t-x) e^{2\pi i\xi t} dx d\xi,
\qquad t \in \Rdst,
\end{align*}
has infinite rank.
\end{lemma}
\begin{proof}
The proof is based on the fact that the STFT of Hermite functions are weighted
polyanalytic functions (cf. \cite{ab10, ab10-1, abgr10}) and therefore cannot vanish on a ball.

Suppose, for the sake of contradiction, that $H_\Omega$ has rank $n-1<+\infty$. Let $h_1, \ldots, h_n \in \LtRd$
be multi-dimensional Hermite functions of order $\leq n$. For example, if $g_1, \ldots, g_n \subseteq L^2(\Rst)$
are the first one-dimensional Hermite functions, we can let $h_k \in \LtRd$ be the tensor product
$h_k(x_1, \ldots, x_d) := g_k(x_1) g_1(x_2) \ldots g_1(x_d)$.
Let $V_n$ be the subspace of $\LtRd$ spanned by $h_1, \ldots, h_n$.

Since $V_n$ has dimension $n$, it follows that there exists some nonzero
$f=\sum_{k=1}^n c_k h_k \in V_n$ such that $H_\Omega f =0$. Consequently,
\begin{align*}
0 = \ip{H_\Omega f}{f} = \int_\Omega \abs{V_\varphi f(z)}^2 dz,
\end{align*}
and therefore $V_\varphi f \equiv 0$ on $\Omega$. With the notation $(x,\xi) =: z \in \bC^n$
and $m(z) = e^{- x \cdot \xi i + \pi \abs{z}^2}$, 
let $F(z) := m(z) V_f \varphi(\overline{z})$.
Since $V_\varphi f(z) = \overline{V_f \varphi(-z)}$, it follows that $F$ vanishes on $\Omega' := -\overline{\Omega}$.
We will show that $F \equiv 0$. Since $m$ never vanishes, this will imply that $f \equiv 0$, thus yielding a
contradiction.

The function $F(z) = \sum_k c_k m(z) V_{h_k} \varphi(-z)$ is a polyanalytic function of order (at most) $n$,
i.e. $(\partial/\partial_{\overline{z}})^\beta F \equiv 0$, for every multi-index $\abs{\beta} \leq n$,
see \cite{ab10, ab10-1, abgr10}. A polyanalytic function
that
vanishes
on a set of non-empty interior must vanish identically. For $d=1$ this can be proved directly by induction on $n$ or
deduced from much sharper uniqueness results (see \cite{ba91-1}). The case of general dimension $d$ reduces to $d=1$ by
fixing $d-1$ variables of $F$ and applying the one-dimensional result.
\end{proof}
We also quote the following intertwining property. 
\begin{lemma}
\label{lemma_cov}
For $z=(x,\xi) \in \Rdst \times \Rdst$, consider the time-frequency shift
\begin{align*}
\pi(z) f(t) := e^{2\pi i \xi t} f(t-x), \qquad f \in \LtRd.
\end{align*}
For $m \in L^\infty(\Rtdst)$, let $H_m: \LtRd \to \LtRd$ be the time-frequency localization operator
with symbol $m$, i.e. $H_m f:= \stft^* (m\stft f)$. Then
\begin{align*}
\pi(z) H_m \pi(z)^* = H_{m(\cdot-z)}, \qquad z=(x,\xi) \in \Rdst \times \Rdst, m \in L^\infty(\Rtdst).
\end{align*}
\end{lemma}
The proof of Lemma \ref{lemma_cov} is a straightforward calculation, see \cite[Lemma 2.6]{dofegr06}.
Using Lemma \ref{lemma_loc_infinite_rank} we obtain a variant of Theorem \ref{th_frames_tf}.

\begin{theorem}
\label{th_frames_tf_noeigen}
Let $w:\Rtdst \to (0,+\infty)$ be an admissible TF weight and let $\varphi \in M^1_w(\Rdst)$.
Let $\sett{\eta_\gamma: \gamma\in\Gamma}$ be a family of non-negative symbols on $\Rtdst$ 
that is well-spread (relative to $w$) and such that $\sum_\gamma \eta_\gamma \approx 1$.
Let $v$ be a $w$-moderated weight.
Assume in addition that there exists a ball $B_r$ and a constant $c>0$ such that
\begin{equation}\label{lowbound}
\eta_\gamma(z) \geq c1_{B_r}(z-\gamma),
\qquad z \in \Rtdst, \gamma \in \Gamma.
\end{equation}
Then there exists a constant $\alpha>0$
such that, for every choice of finite subsets of eigenfunctions
$\set{\phigk}{\gamma \in \Gamma, 1 \leq k \leq N_\gamma}$ 
with $\inf_\gamma N_\gamma \geq \alpha$ and $\sup_\gamma N_\gamma < +\infty$,
the following frame estimates hold simultaneously
for all $1 \leq p \leq +\infty$, with the usual modification for $p= \infty$:
\begin{align}\label{inequnweight}
\norm{f}_{\Mpv} \approx
\left(
\sum_{\gamma\in\Gamma} \sum_{k=1}^{N_\gamma} \abs{\ip{f}{\phi^{\gamma}_k}}^p v(\gamma)^p
\right)^{1/p},
\qquad f \in M^p_v(\Rdst).
\end{align}
Moreover, $\alpha$ can be chosen uniformly for any class of weights $v$ having a uniform constant $C_v$
(cf. \eqref{eq_v_moderated}).
\end{theorem}
Before proving Theorem \ref{th_frames_tf_noeigen} we make some remarks.
\begin{rem}
When $\eta_\gamma$ is the characteristic function of a set $\Omega_\gamma$, the condition 
in \eqref{lowbound} holds whenever the sets satisfy:
$B_r(\gamma) \subseteq \Omega_\gamma \subseteq B_R(\gamma)$,
with $R>r>0$ and $\Gamma$ a relatively separated set.
\end{rem}

\begin{rem}
The frame in Theorem \ref{th_frames_tf_noeigen} comprises the first $N_\gamma$ elements of each of the orthonormal
sets $\sett{\phigk: k \geq 1}$. These first $N_\gamma$ functions are the ones that are best concentrated, according
to the weight $\eta_\gamma$.
This resembles the problem studied in \cite{ro11}. However the results there require very precise information on the
frames
being pieced together and hence do not apply here.
\end{rem}
\begin{rem}
In the language of \cite{os94-1, caku04}, 
Theorem \ref{th_frames_tf_noeigen} shows that the subspaces spanned by the finite families of
eigenfunctions form a \emph{stable splitting} or \emph{fusion frame}. From an application point of view, it is useful to
have orthogonal projections onto subspaces with time-frequency concentration in a prescribed area of the
time-frequency
plane.
\end{rem}
\begin{rem}
When $\eta_\gamma$ is the characteristic function of a set $\Omega+\gamma$ and $\Gamma$ is a lattice, then
Theorem \ref{th_frames_tf_noeigen} reduces to the main technical result in \cite{dogr11}. The proof there does not adapt
to the irregular context, since it relies on the use of rotation algebras (non-commutative tori). The proof we give
here
resorts instead to the almost-orthogonality techniques from \cite{ro12} (cf. Theorem \ref{th_almost_orth}) together with
spectral invariance results for pseudodifferential operators with symbols in the Sj\"ostrand class
\cite{{sj94,sj95,gr06,gr06-3}}.
\end{rem}

\begin{proof}[Proof of Theorem~\ref{th_frames_tf_noeigen}]
First note that \eqref{lowbound}, together with the well-spreadness condition, implies that $\norm{\eta_\gamma}_1
\approx
1$, (the constants, of course, depend on $r,R$ and $c$). Hence the
condition $\alpha \norm{\eta_\gamma}_1 \leq N_\gamma$ required by Theorem \ref{th_frames_tf} can be
granted by simply requiring $\tilde{\alpha} \leq N_\gamma$ to hold with a different constant $\tilde{\alpha}$.

By Theorem~\ref{th_frames_tf} we have that
$\norm{f}_{\Mpv} \approx
\Big(
\sum_{\gamma\in\Gamma} \sum_{k=1}^{N_\gamma} \abs{\ip{f}{ \lambdagk \phi^{\gamma}_k}}^p v(\gamma)^p
\Big)^{1/p}$.
Hence, it suffices to show that $\lambdagk \approx 1$, for $1 \leq k \leq N_\gamma$.

The upper bound follows from the well-spreadness condition, because if $(\Gamma, \env, g)$ is an envelope
for $\sett{H_{\eta_\gamma}: \gamma \in \Gamma}$, then
all the singular values of $H_{\eta_\gamma}$ are bounded
by $\norm{H_{\eta_\gamma}}_{2 \to 2} \leq \|g\|_\infty$, cf. \eqref{eq_bound_M}.

Let $N:= \sup_\gamma N_\gamma$. By Lemma \ref{lemma_loc_infinite_rank}, the localization operator $H_{B_r}$ has
infinite rank. Hence, the non-zero eigenvalues of $H_{B_r}$ form an infinite non-increasing sequence
$\lambda^R_k>0$, $k \geq 1$. From \eqref{lowbound} it follows that $\locmg \geq c H_{B_r + \gamma}$ 
(cf. Proposition \ref{prop_mult_hilbert}). By Lemma \ref{lemma_cov}, the sequence of
eigenvalues of $H_{B_r+\gamma}$ coincides with the one of $H_{B_r}$.
Consequently, for $1 \leq k \leq N_\gamma$
\begin{align*}
\lambdagk \geq c \lambda^R_k \geq c \lambda^R_{N_\gamma}\geq c \lambda^R_N >0,
\end{align*}
as desired.
\end{proof}

\begin{rem}
For the results in this section, the abstract setting of Section \ref{sec_abstract} allows for the replacement of
$\ellpv$ by
more general normed spaces. Indeed, the results derived in the abstract setting cover modulation spaces 
defined with respect to general translation-invariant solid spaces, cf.~\cite{fegr89}.
\end{rem}

\section{Frames of eigenfunctions: general $L^2$ estimates}
\label{sec_abs_frames}
In this section we prove results similar to the ones in Section \ref{sec_frames_tf}, but this time in the abstract
setting of Section~\ref{sec_abs_mol}. We work only with the space $L^2$ instead of treating a class of Banach
spaces. The reason for this restriction is that in Section~\ref{sec_frames_tf} we used tools from 
the theory of pseudodifferential operators to extend certain results from $L^2(\Rdst)$ to a range of modulation spaces,
and those tools are not available in the abstract setting. The proofs in this section are, mutatis mutandis, the same as
in Section \ref{sec_tf} and will be just sketched.

Let $\Gc$ be a locally compact, $\sigma$-compact group. Let $\Esp=L^2(\Gc)$ and let us assume that (A1), (A2) and
(A3) from Section \ref{sec_model} hold. Let a well-spread family $\sett{\eta_\gamma: \gamma \in \Gamma}$ of non-negative
functions
on $\Gc$ be given. We consider the corresponding family of phase-space multipliers,
\begin{align*}
\multmg f = P(\eta_\gamma f), \quad f \in \SLt.
\end{align*}
Since each $\eta_\gamma$ is non-negative and belongs to $\Lisp(\Gc)$, according to Proposition~\ref{prop_mult_hilbert},
the corresponding operator $\multmg: \SLt \to \SLt$ is positive and trace-class, and $\trace(\multmg) \lesssim
\norm{\eta_\gamma}_1$.
Let $\multmg: \SLt \to \SLt$ be diagonalized as
\begin{align}
\label{eq_abs_diag}
\multmg f = \sum_{k \geq 1} \lambdagk \ip{f}{\phigk}{\phigk},
\qquad f \in \SLt,
\end{align}
where $\set{\phigk}{k \geq 1}$ is an orthonormal subset of $\SLt$ 
and $\set{\lambdagk}{k \geq 1}$ is decreasing. Let us define
\begin{align*}
\multmgeps f =
\sum_{k:\lambdagk>\varepsilon}
\lambdagk \ip{f}{\phigk}{\phigk}, \qquad f \in \SLt.
\end{align*}
If $\multmg$ has finite rank, then $\lambdagk = 0$, for $k \gg 1$ and the choice
of the corresponding eigenfunctions is arbitrary.

We now derive results similar to Theorems~\ref{th_finite_rank_tf} and ~\ref{th_frames_tf}, but this time in the 
current abstract setting.

\begin{theorem}
\label{th_finite_rank_abs_Lt}
Under Assumptions (A1), (A2) and (A3), let $\sett{\eta_\gamma: \gamma \in \Gamma}$ be a well-spread family of
non-negative symbols such that $\sum_\gamma \eta_\gamma \approx 1$. Then there exist constants $0<c \leq C < +\infty$
such that for all sufficiently small $\varepsilon >0$,
\begin{align*}
c \norm{f}_2^2 \leq \sum_{\gamma\in\Gamma} \norm{\multmgeps f}_2^2  \leq C \norm{f}_2^2,
\qquad f \in \SLt.
\end{align*}
Furthermore, there exists a constant $\alpha>0$ such that, for every choice of
numbers $\{N_\gamma: \gamma \in \Gamma\} \subseteq \Nst$ satisfying
$\alpha \norm{\eta_\gamma}_1 \leq N_\gamma$ and $\sup_\gamma N_\gamma < +\infty$,
the family
\begin{align}
\label{eq_frame_abs}
\set{\lambdagk \phigk}{\gamma \in \Gamma, 1 \leq k \leq N_\gamma},
\end{align}
formed from eigenfunctions and eigenvalues of the operator $\multmg$, is a frame of $\SLt$.
\end{theorem}
Before proving Theorem \ref{th_finite_rank_abs_Lt} we make some remarks.
\begin{rem}
Note again that when $\eta_\gamma$ is the characteristic function of a set $\Omega_\gamma$, we are picking
$\approx \abs{\Omega_\gamma}$ eigenfunctions from each phase-space multiplier.
Here, $\abs{\Omega_\gamma}$ is the Haar measure of $\Omega_\gamma$.
\end{rem}
\begin{rem}
The operator $\multmg$ may have finite rank (for example if $\Gc$ is a discrete group and $\eta_\gamma$ is
the characteristic function of a finite set). In this case the choice of the eigenfunctions associated to the singular
value
zero is irrelevant, since in \eqref{eq_frame_abs} these are multiplied by zero.
\end{rem}
\begin{proof}[Proof of Theorem \ref{th_finite_rank_abs_Lt}]
We parallel the proofs in Section \ref{sec_frames_tf}.
Since $\sum_\gamma \multmg = M_{\sum_\gamma \eta_\gamma}$ and $\sum_\gamma \eta_\gamma \geq A>0$, it follows 
from Proposition \ref{prop_mult_hilbert} that $\sum_\gamma \multmg$ is positive definite and therefore invertible.
Theorem \ref{th_almost_orth} consequently yields,
\begin{align}
\sum_{\gamma\in\Gamma} \norm{\multmg f}_{L^2(\Gc)}^2
\approx
\norm{f}_2^2,
\qquad f \in \SLt.
\end{align}
In addition, Proposition~\ref{prop_squared_covers_are_molecules}, Corollary \ref{coro_squares} 
and a second application of Theorem \ref{th_almost_orth} yield
\begin{align*}
\sum_{\gamma\in\Gamma} \norm{\multmg^2 f}_{L^2(\Gc)}^2
\approx
\norm{f}_2^2,
\qquad f \in \SLt.
\end{align*}
The thresholded operators $\multmgeps$ satisfy
\begin{align*}
\norm{\multmgeps f}_2 \leq \norm{\multmg f}_2 \leq \norm{\multmgeps f}_2 + \varepsilon \norm{f}_2,
\qquad f \in \SLt.
\end{align*}
Applying this to $\multmg f$ and noting that $\multmg$ and $\multmgeps$ commute gives
\begin{align*}
\norm{\multmg^2 f}_2 &\leq \norm{\multmgeps \multmg f}_2 + \varepsilon \norm{\multmg f}_2,
\\
&= \norm{\multmg \multmgeps f}_2 + \varepsilon \norm{\multmg f}_2,
\\
&\lesssim 
\norm{\multmgeps f}_2 + \varepsilon \norm{\multmg f}_2,
\qquad f \in \SLt.
\end{align*}
Putting all these inequalities together gives
\begin{align*}
\big(\sum_{\gamma\in\Gamma} \norm{\multmgeps f}^2_2 \big)^{1/2}
\lesssim \norm{f}_2
\lesssim \big(\sum_{\gamma\in\Gamma} \norm{\multmgeps f}^2_2 \big)^{1/2} + \varepsilon \norm{f}_2,
\qquad f \in \SLt.
\end{align*}
This implies that, for $0<\varepsilon \ll 1$,
$\big(\sum_{\gamma\in\Gamma} \norm{\multmgeps f}^2_2 \big)^{1/2}
\approx \norm{f}_2$,
as claimed. The fact that the system in \eqref{eq_frame_abs} is a frame of $\SLt$ now follows like
in Theorem \ref{th_frames_tf}, this time using Proposition~\ref{prop_mult_hilbert} to estimate:
$\# \sett{\lambdagk: \lambdagk > \varepsilon} \leq \varepsilon^{-1} \trace(\multmg)
\lesssim \varepsilon^{-1} \norm{\eta_\gamma}_1$.
\end{proof}

\subsection{Application to time-scale analysis}
\label{sec_wavelets}
We now show how to apply Theorem~\ref{th_finite_rank_abs_Lt}  to time-scale analysis.
Let $\psi: \Rdst \to \bC$ be a Schwartz-class radial function with several vanishing moments. The \emph{wavelet
transform} of a function $f \in \LtRd$ with respect to $\psi$ is defined by
\begin{align}
W_\psi f(x,s) = s^{-d/2}\int_{\Rdst} f(t)
\overline{\psi \left(\frac{t-x}{s}\right)}dt,
\qquad x \in \Rdst, s>0.
\end{align}
If $\psi$ is properly normalized (and we assume so thereof), $W_\psi$ maps $\LtRd$ isometrically into
$\Ltsp(\Rdst \times \Rst_{+}, s^{-(d+1)}dxds)$. For a measurable bounded
symbol $m: \Rdst \times \Rst_{+} \to \bC$, the \emph{wavelet multiplier}
$\mathrm{WM}_m: \LtRd \to \LtRd$ is defined as
\begin{align}
\label{eq_deff_WM}
\mathrm{WM}_m f(t)
= \int_0^{+\infty} \int_\Rdst
m(x,s) W_\psi f(x,s) \pi(x,s) \psi (t) dx \frac{ds}{s^{d+1}}, \qquad t \in \Rdst,
\end{align}
where $\pi(x,s) \psi (t) := s^{-d/2} \psi \left(\frac{t-x}{s}\right)$.
(Here, the integral converges in the weak sense.) Note that $\mathrm{WM}_m=W^*_\psi(m W_\psi F)$.
The operator $\mathrm{WM}_m$ is also known as \emph{wavelet localization operator} \cite{da88,dapa88,da90}.

In order to apply the model from Section \ref{sec_model} we consider the affine group
$\Gc = \Rdst \times \Rst_{+}$,
where multiplication is given by $(x,s)\cdot(x',s') = (x+sx',ss')$.
The Haar measure in $\Gc$ is given by $\abs{X} = \int_X s^{-(d+1)} dxds$ and the modular
function is given by $\Delta(x,s)=s^{-d}$.

We let $\Esp:=\Ltsp(\Gc)$ and $\SEsp := W_\psi \LtRd$. In complete analogy to the
time-frequency analysis case, we let $P:= W_\psi W^*_\psi: \Ltsp(\Gc) \to W_\psi \LtRd$
be the orthogonal projection and $\kernelenv := \abs{W_\psi \psi}$. We further
let $w(x,s) := \max\sett{1, s^d}$. The kernel $\kernelenv$
belongs to $\win \cap \winr$ if $\psi$ has sufficiently many vanishing moments (see
\cite[Section 4.2]{grpi09}).

As an example of a well-spread family of symbols we consider the characteristic functions
of a cover of $\Gc$ by irregular boxes. Let us take as centers the points
\begin{align*}
\Gamma := \set{\gamma_{j,l} := (l2^j,2^j)}{j \in \Zst, l\in\Zdst},
\end{align*}
and consider a family of boxes around $(0,1) \in \Rdst\times\Rst_{+}$,
\begin{align}
\label{eq_set_1}
V_{j,l} := [-a_{j,l}^1 /2,a_{j,l}^1 /2] \times \ldots \times [-a_{j,l}^d /2,a_{j,l}^d /2]
\times [(b_{j,l})^{-1},b_{j,l}],
\end{align}
where $0 \leq a^i_{j,l} \leq a < +\infty$, $i = 1,\ldots , d$ and
$0 < b^{-1} \leq b_{j,l} \leq b < +\infty$.
Let us set
\begin{align}
\label{eq_set_2}
U_{j,l} := \gamma_{j,l} V_{j,l}, \qquad j \in \Zst, l\in\Zdst.
\end{align}
The family of characteristic functions $\sett{1_{U_{j,l}}: j \in \Zst, l\in\Zdst}$ is well-spread,
with envelope $(\Gamma,g)$, where $\Gamma := \set{\gamma_{j,l} := (l2^j,2^j)}{j \in \Zst, l\in\Zdst}$
and $g:=1_{[-a/2,a/2]^d \times [b^{-1},b]}$.

Note that
$\norm{1_{U_{j,l}}}_{L^1(\Gc)} = \mes{U_{j,l}} = \mes{V_{j,l}} =
\frac{1}{d}\prod_{i = 1}^d a_{j,l}^i \cdot [(b_{j,l})^d-({b_{j,l}})^{-d}]$.
Theorem~\ref{th_finite_rank_abs_Lt} yields the following.

\begin{theorem}\label{Th:WavCase}
Suppose that the sets $\set{U_{j,l}}{j\in\Zst,l\in\Zdst}$ from \eqref{eq_set_1} and \eqref{eq_set_2}
cover $\Rdst \times \Rst_{+}$.

For each $j\in\Zst,l\in\Zdst$ let the wavelet multiplier $WM_{1_{U_{j,l}}}$ - cf. \eqref{eq_deff_WM} -
have eigenfunctions $\{\phi^k_{j,l}: k \geq 1\}$ with corresponding
eigenvalues $\{\lambda^k_{j,l}: k \geq 1\}$ ordered decreasingly.

Then there exists a constant $\alpha>0$ such for every choice of numbers $\{N_{j,l}: j\in\Zst,l\in\Zdst\} \subseteq
\Nst$
satisfying
\begin{align*}
\alpha \mes{U_{j,l}} \leq N_{j,l}
\mbox { and }
\sup_{j,l} N_{j,l} < +\infty,
\end{align*}
the family $\set{\lambda^k_{j,l} \phi^k_{j,l}}{j\in\Zst,l\in\Zdst, 1 \leq k \leq N_{j,l}}$
is a frame of $\LtRd$.
\end{theorem}
\subsection{Application to Gabor analysis}
\label{sec_gabor}
Let us consider a window $\varphi\in M^1 (\mathbb{R}^d)$ with $\norm{\varphi}_2=1$ and a
(full rank) lattice $\Lambda\subseteq\mathbb{R}^{2d}$, i.e. $\Lambda=P \Zst^{2d}$, where $P \in \Rst^{2d
\times 2d}$ is an invertible matrix. The \emph{Gabor system} associated with $\varphi$ and $\Lambda$ is the collection
of functions
\begin{align*}
\mathcal{G}_{\varphi, \Lambda} :=
\set{\varphi_\lambda (t):=e^{2\pi i \xi t}\varphi(t-x)}{\lambda=(x,\xi) \in \Lambda}.
\end{align*}
We assume that this collection of functions is a \emph{tight-frame}. This means that 
for some constant $A>0$, every function $f \in L^2(\Rdst)$ admits the expansion
\begin{align}
\label{eq_gab_exp}
f=A  \sum_{\lambda\in\Lambda}  \ip{f}{\pi(\lambda) \varphi}{\pi(\lambda ) \varphi}.
\end{align}
In this section we show how to apply the abstract results from Section \ref{sec_abs_frames} to obtain frames 
consisting of functions $f$ whose coefficients $\ip{f}{\pi(\lambda) \varphi}$ have a prescribed profile.

For a bounded sequence $m:\Lambda\rightarrow \bC$, the \emph{Gabor multiplier} $GM_{m}: L^2(\mathbb{R}^d ) \rightarrow
L^2(\mathbb{R}^d )$ is defined by applying the mask $m$ to the frame expansion in \eqref{eq_gab_exp}
\begin{align}
\label{eq_gab_mul}
GM_{m} f=A \sum_{\lambda\in\Lambda}  m(\lambda) \ip{f}{\pi(\lambda) \varphi}{\pi(\lambda ) \varphi}.
\end{align}
(See ~\cite{feno03} for a survey on Gabor multiplier; see also \cite{doto10,gr11-2}.)
If $m \geq 0$, the first $N$ eigenfunctions of $GM_{m}$ form an orthonormal set in
$\LtRd$ that maximizes the quantity
\begin{align*}
\sum_{j = 1}^N  \sum_{\lambda\in\Lambda} m(\lambda) \abs{\ip{f_j}{\pi(\lambda) \varphi}}^2
\end{align*}
among all orthonormal sets of $N$ functions $\{f_1, \ldots, f_N\} \subseteq \LtRd$.

Let us show how the abstract setting of Section \ref{sec_model} can be applied. The discussion is analogous to
Example \ref{ex_tf}. We let $\Gc=\Lambda$, considered as
a group and $\Esp := \ell^2(\Lambda)$. Consider the \emph{analysis operator} $C_{\Lambda,\varphi}:
L^2(\mathbb{R}^d) \rightarrow \ell^2 (\Lambda )$ given by $ C_{\Lambda,\varphi} f(\lambda ) = \sqrt{A}\langle f
, \pi (\lambda ) \varphi \rangle$. Let $S_{\ell^2} := C_{\Lambda,\varphi}( L^2(\mathbb{R}^d))$. Since we assume
that $\mathcal{G}_{\varphi, \Lambda}$ is a tight frame, the operator
$C_{\Lambda,\varphi}: L^2(\mathbb{R}^d) \rightarrow \ell^2 (\Lambda)$ is an isometry
- cf. \eqref{eq_gab_exp}. The
orthogonal projection $P: \ell^2(\Lambda) \to S_{\ell^2}$ is then $P=C_{\Lambda,\varphi} 
C^\ast_{\Lambda,\varphi}$ and is therefore represented by the matrix $ \kappa(\mu, \lambda) =
A\langle \varphi_\lambda, \varphi_\mu \rangle$. 
Consequently,
\begin{equation*}
|\kappa(\mu, \lambda )| = A|\langle \varphi_\lambda, \varphi_\mu \rangle|
=A| \Vstft_{\varphi} \varphi(\mu-\lambda)|,
\quad \mu, \lambda \in \Lambda.
\end{equation*}
Since $\varphi \in M^1 (\mathbb{R}^d)$, $C_{\Lambda, \varphi}$ maps $M^1 (\mathbb{R}^d)$ into
$\ell^1(\Lambda)$ (see for example~\cite{fezi98}) and we conclude that
$\kernelenv := A\abs{\Vstft_{\varphi} \varphi _{|\Lambda}}=
\sqrt{A} \abs{C_{\Lambda, \varphi} \varphi} \in \ell^1(\Lambda)=W(\ell^\infty,\ell^1)(\Lambda)$.
Hence, (A1), (A2) and (A3) from Section
\ref{sec_abstract} are satisfied with $\Gc=\Lambda$, $\Esp=\ell^2(\Lambda)$ and $w\equiv 1$.

In addition, note that the Gabor multiplier in \eqref{eq_gab_mul} satisfies
$GM_{m} f = C^*_{\Lambda,\varphi} (m C_{\Lambda,\varphi} f )$.
Hence $C_{\Lambda,\varphi} GM_{m} C^*_{\Lambda,\varphi}: S_{\ell^2} \to S_{\ell^2}$
is a phase-space multiplier with symbol $m$ - cf. Section \ref{sec_phmul}.

As an example of a well-spread family of symbols on $\Lambda$ we may now consider a well-spread family
$\sett{\eta_\gamma: \gamma \in \Gamma}$ of non-negative symbols defined on $\Rst^{2d}$, where $\Gamma \subseteq \Lambda$
is a relatively separated subset of $\Rdst$ and $\sum_\gamma \eta_\gamma \approx 1$, and restrict each $\eta_\gamma$ to
$\Lambda$. As an application of \ref{th_finite_rank_abs_Lt} we obtain the following result.

\begin{theorem}\label{Th:GabMul}
Let $\varphi\in M^1 (\mathbb{R}^d)$ with $\norm{\varphi}_2=1$ and $\Lambda \subseteq \Rtdst$ be a lattice.
Let $\sett{\eta_\gamma: \gamma \in \Gamma}$ be a well-spread family of non-negative symbols defined on
$\Rst^{2d}$ such that $\sum_\gamma \eta_\gamma \approx 1$. Let us restrict each $\eta_\gamma$ to $\Lambda$
and consider the corresponding Gabor multiplier $GM_{\eta_\gamma}$ - cf. \eqref{eq_gab_mul} - having eigenfunctions
$\sett{\phigk: k \geq 1}$ in decreasing order with respect to the corresponding eigenvalues $\lambdagk$.

Then there exists a constant $\alpha>0$ such for every choice of 
numbers $\{N_\gamma: \gamma \in \Gamma\} \subseteq \Nst$ satisfying
\begin{align*}
\alpha \|\eta_\gamma\|_{\ell^1(\Lambda)} \leq N_\gamma
\mbox{ and }
\sup_\gamma N_\gamma < +\infty,
\end{align*}
the family
$\set{\lambdagk\phigk}{\gamma \in \Gamma, 1 \leq k \leq N_\gamma}$
is a frame of $\LtRd$.
\end{theorem}
\begin{rem}
While Theorem \ref{th_frames_tf} provides frames for $\LtRd$ consisting of functions having a
spectrogram that is optimally adapted to a given weight on $\Rtdst$
- cf. \eqref{eq_optimal}, Theorem \ref{Th:GabMul} provides frame elements where the profile of the coefficients
associated with the discrete expansion in \eqref{eq_gab_exp} is optimized with respect to a weight on $\Lambda$.
\end{rem}

\section*{Appendix A: proof of Theorem \ref{th_almost_orth}}
\label{app_A}
In this appendix we prove Theorem \ref{th_almost_orth}. The proof is essentially contained in \cite{ro12},
but is not explicitly stated in the required generality. We therefore show how to derive Theorem \ref{th_almost_orth}
from some technical lemmas in \cite{ro12}.
\begin{rem}
We quote simplified versions of some statements in \cite{ro12}. The article \cite{ro12} considers a technical
variant of the amalgam space $\winr(\Gc)$, called the weak amalgam space $\wweak$ (see \cite[Section 2.4]{ro12}), which
we do not wish to introduce here. By \cite[Proposition 1]{ro12}, $\Lisp(\Gc) \into \wweak (\Gc) \into \winr(\Gc)$. Some
results from \cite{ro12} that we quote assume that a certain function $g$ belongs to $\winr(\Gc)$ and are proved
in \cite{ro12} under the weaker assumption: $g \in \wweak$.
\end{rem}

We quote the following estimate.
\begin{lemma}\cite[Lemma  3.8]{fegr89}\cite[Lemma 2]{ro12}
\label{lemma_amfact}
Let $\Esp$ be a solid, translation invariant BF space, let $w$ be an admissible weight for it and let
$\Gamma \subseteq \Gc$ be a relatively separated set.
Then for every $f \in \Esp$ and $g \in W_R(C_0,L^1_w)$,
the sequence $(\ip{f}{L_\lambda g})_{\lambda \in \Lambda}$
belongs to $\Espd(\Lambda)$ and
\begin{align*}
\norm{(\ip{f}{L_\lambda g})_\lambda}_{\Espd}
\lesssim \norm{f}_\Esp \norm{g}_{W_R(L^\infty,L^1_w)}. 
\end{align*}
The implicit constants depend on the spreadness $\rel(\Gamma)$ - cf. \eqref{eq_spread}.
\end{lemma}

Suppose that Assumptions (A1) and (A2) from Section \ref{sec_model} hold.

For a solid, translation invariant BF space $\Esp$ we 
consider an $\Ltsp$-valued version of $\Espd(\Gamma)$,
\begin{align*}
\EspdLt=\EspdLt(\Gamma) := \set{(f_\gamma)_{\gamma \in \Gamma} \in (\Ltsp(\Gc))^\Gamma}
{ (\norm{f_\gamma}_\Ltsp)_{\gamma \in \Gamma} \in \Espd(\Gamma)},
\end{align*}
and endow it with the norm
$\norm{(f_\gamma)_{\gamma \in \Gamma}}_{\EspdLt} :=
\norm{(\norm{f_\gamma}_{\Ltsp})_{\gamma \in \Gamma}}_{\Espd}$.

Let $\set{\molg}{\gamma \in \Gamma}$ be a well-spread family of operators - cf. Section \ref{sec_abs_mol}.
Let $U \subseteq \Gc$ be a relatively compact neighborhood
of the identity. Consider the operators $\coeft$ and $\recv$ formally defined by
\begin{align}
\label{eq_deff_A}
\coeft(f) &:= (\molg(f))_{\gamma \in \Gamma},
\qquad f \in \SEsp,
\\
\label{eq_deff_B}
\recv((f_\gamma)_{\gamma \in \Gamma}) &:= \sum_{\gamma\in\Gamma} P(f_\gamma) 1_{\gamma U},
\qquad f_\gamma \in \Ltsp(\Gc),
\end{align}
where $1_{\gamma U}$ denotes the characteristic function of the set $\gamma U$. These operators satisfy the following
mapping properties.
\begin{prop}
\label{prop_vector_synt}
Assume (A1) and (A2) and let $\set{\molg}{\gamma \in \Gamma}$ be a well-spread family of operators.
Then the operators $\coeft$ and $\recv$ in \eqref{eq_deff_A} and
\eqref{eq_deff_B} satisfy the following.
\begin{itemize}
 \item[(a)]
The analysis operator $\coeft$ maps $\SEsp$ boundedly into
$\EspdLt(\Gamma)$.
\item[(b)]
For every relatively compact neighborhood of the identity $U$,
and every sequence $F \equiv (f_\gamma)_\gamma \in \EspdLt$,
the series defining $\recv(F)$ converge absolutely in $\Ltsp(\Gc)$ at
every point. Moreover, the operator $\recv$ maps 
$\EspdLt(\Gamma)$ boundedly into $\Esp$
(with a bound that depends on U). 
\end{itemize}
\end{prop}
\begin{proof}
Part (b) is proved in \cite[Proposition 4 (b)]{ro12} under weaker hypothesis.
Part (a) is a slight variant of \cite[Proposition 4 (a)]{ro12}; for completeness we give a full argument.
Let $(\Gamma, \env, g)$ be an envelope for $\set{\molg}{\gamma \in \Gamma}$.

Let $f \in \SEsp$. Since $\eta_\gamma$ is bounded, $f \eta_\gamma \in \Esp$.
By the definition of well-spread family - cf. \eqref{eq_deff_loc_mol}
\begin{align*}
\abs{\molg f(x)} &\leq
\int_\Gc \abs{f(y)} g(\gamma^{-1}y) \env(y^{-1}x) dy
= \left( \abs{f}L_\gamma g \right) * \env (x).
\end{align*}
By Young's inequality $L^1* L^2 \hookrightarrow L^2$ we have
\begin{align*}
\norm{\molg f}_2 \leq \norm{\env}_2
\int_\Gc \abs{f(y)} g(\gamma^{-1}y) dy
\lesssim
\norm{\env}_{W(L^\infty,L^1_w)}
\int_\Gc \abs{f(y)} g(\gamma^{-1}y) dy.
\end{align*}
Now Lemma \ref{lemma_amfact} yields
$ \norm{\coeft(f)}_{\EspdLt}\lesssim \norm{f}_{\Esp} \norm{g}_\winr$, as desired.
\end{proof}
\begin{rem}
Note that in the last proof the use of the $L^2$ norm is somewhat arbitrary; a number of other function norms could 
have been used instead (cf. \cite[Proposition 4]{ro12}).
\end{rem}
Now we prove the key approximation result (cf. \cite[Theorem 1]{ro12}).
\begin{theorem}
\label{th_approx}
Assume (A1) and (A2) and let $\set{\molg}{\gamma \in \Gamma}$ be a well-spread family of operators.
Given $\varepsilon >0$, there exists $U_0$, a relatively compact neighborhood
of e such that for all $U \supseteq U_0$
\begin{align}
\label{eq_th_approx}
 \norm{\sum_{\gamma\in\Gamma} T_\gamma f - \recv \coeft(f)}_\Esp \leq \varepsilon \norm{f}_\Esp,
\quad f \in \SEsp.
\end{align}
\end{theorem}
\begin{rem}
The neighborhood $U_0$ can be chosen uniformly for any class of spaces
$\Esp$ having the same weight $w$ and the same constant $\consEw$ (cf.
\eqref{weight_w_admissible}).

Concerning the parameters in Assumptions (A1) and (A2)
and \eqref{eq_deff_loc_mol}, the choice of $U_0$ only depends on
$\norm{\kernelenv}_{W(L^\infty,L^1_w)}$,
$\norm{\kernelenv}_{W_R(L^\infty,L^1_w)}$,
$\norm{\env}_{W(L^\infty,L^1_w)}$,
$\norm{\env}_{W_R(L^\infty,L^1_w)}$,
$\norm{g}_{\winr}$ and
$\rel(\Gamma)$ (cf. \eqref{eq_spread}).
\end{rem}
\begin{proof}[Proof of Theorem \ref{th_approx}]
Let $f \in \SEsp$ and let $U$ be a relatively compact neighborhood of e. Because of the inclusion
$\SEsp \into W(L^\infty,\Esp)$ in Proposition \ref{prop_P_into_am}, it suffices to dominate the left-hand side
of \eqref{eq_th_approx} by $\varepsilon \norm{f}_{W(L^\infty,\Esp)}$.

Note that since $\molg f \in \SEsp$,
$\recv\coeft f(x) = \sum_{\gamma\in\Gamma} \molg f(x) 1_{\gamma U}(x)$. Hence,
using \eqref{eq_deff_loc_mol} let us estimate
\begin{align*}
&\abs{\sum_{\gamma\in\Gamma} \molg f(x) - \recv\coeft f(x)} =
\abs{\sum_{\gamma\in\Gamma} 1_{\gamma (\Gc \setminus U)}(x) \molg(f)(x)}
\\
&\qquad\leq \sum_{\gamma\in\Gamma} \int_{\Gc} \abs{f(y)} g(\gamma^{-1}y) \env(y^{-1}x)
1_{\gamma(\Gc \setminus U)}(x) dy.
\end{align*}
The rest of  the proof is carried out exactly as in \cite[Theorem 1]{ro12}. Indeed, the proof there only depends
on the  estimate just derived.\footnote{The function $\env$ is called $H$ in the proof \cite[Theorem 1]{ro12}.}
(The definition of well-spread family of operators was tailored so that the proof in \cite[Theorem 1]{ro12}
would still work.)
\end{proof}
Finally we can prove Theorem \ref{th_almost_orth}.
\begin{proof}[Proof of Theorem \ref{th_almost_orth}]
Let $\sett{\molg: \gamma \in \Gamma}$ be a well-spread family of operators and suppose that the operator
$\sum_\gamma \molg: \SEsp \to \SEsp$ is invertible. We have to show that for $f \in \SEsp$,
$\norm{f}_\Esp \approx \norm{\coeft(f)}_{\EspdLt(\Gamma)}$. The estimate $\norm{\coeft(f)}_{\EspdLt(\Gamma)} \lesssim
\norm{f}_\Esp$
is proved in Proposition~\ref{prop_vector_synt} (a). To establish the second inequality,  consider the operator
$P\recv\coeft: \SEsp \to \SEsp$. Then for $f \in \SEsp$
\begin{align*}
\norm{\sum_{\gamma\in\Gamma} \molg f - P \recv\coeft f}_\Esp =
\norm{P \sum_{\gamma\in\Gamma} \molg f -P\recv\coeft f }_\Esp
\lesssim \norm{\sum_{\gamma\in\Gamma} \molg f -\recv\coeft f}_\Esp. 
\end{align*}
This estimate, together with Theorem~\ref{th_approx} implies that
$\norm{\sum_\gamma \molg - P \recv\coeft}_{\SEsp \to \SEsp} \rightarrow 0$
as $U$ grows to $\Gc$.  Hence, there exists $U$ such that $P\recv\coeft$ is invertible on $\SEsp$.
Consequently, for $f \in \SEsp$, $\norm{f}_\Esp \approx \norm{P\recv\coeft f}_\Esp \lesssim
\norm{\coeft(f)}_{\EspdLt(\Gamma)}$.
Here we have used the boundedness of $\recv$ - contained in Proposition~\ref{prop_vector_synt} (b) - and the boundedness
of $P:\Esp \to W(L^\infty,\Esp) \into \Esp$ - contained in Proposition~\ref{prop_P_into_am}.
\end{proof}

\section*{Appendix B: Pseudodifferential operators and proof of Proposition \ref{prop_wiener}}
\label{app_B}
The Weyl transform of a distribution $\sigma \in \mathcal{S}'(\Rdst \times \Rdst)$ is 
an operator $\sigma^w$ that is formally defined on functions $f:\Rdst \to \bC$ as
\begin{align}
\label{eq_weyl}
\sigma^w (f)(x) := \int_{\Rdst \times \Rdst} \sigma\left(\frac{x+y}{2},\xi\right) e^{2\pi i(x-y)\xi} f(y) dy d\xi,
\qquad x \in \Rdst.
\end{align}
The fundamental results in the theory of pseudodifferential operators provide conditions on $\sigma$
for the operator $\sigma^w$ to be well-defined and bounded on various function spaces.
We now quote some results about pseudodifferential operators acting on modulation spaces - cf. Section
\ref{sec_mod_sp}.

In \cite{gr06-3,gr06} it was shown that modulation spaces on $\Rtdst$ serve as symbol classes to study
pseudodifferential operators acting on modulation spaces on $\Rdst$, recovering and extending classical results from
Sj\"ostrand \cite{sj94,sj95}. We quote the following simplified version of \cite[Theorems 4.1 and 4.6, and Corollaries
3.3 and 4.7]{gr06}. (The GRS condition for admissible weights in Section \ref{sec_mod_sp} is important here.)

\begin{theorem}
\label{th_weyl}
Let $w$ be an admissible TF weight - cf. Definition \ref{def_tf_weights} - and let $\varphi \in M^1_w(\Rdst)$ be
non-zero. Let us denote $\widetilde{w}(z_1,z_2) = w(-z_2,z_1)$. Then the following statements hold true.
\begin{itemize}
\item[(i)] If $\sigma \in \Sjclass$, then $\sigma^w$ is bounded on $M^p_v(\Rdst)$, for all $w$-moderated weights $v$
and all $p \in [1,+\infty]$.
\item[(ii)] If $\sigma \in \Sjclass$ and $\sigma^w$ is invertible as an operator on $\LtRd$, then
$\sigma^w$ is invertible as an operators on $M^p_v(\Rdst)$, for all $w$-moderated weights $v$
and all $p \in [1,+\infty]$.
\item[(iii)] Let $T:\mathcal{S}(\Rdst)\to\mathcal{S}'(\Rdst)$ be a linear and continuous operator.
For $(x,\xi) \in \Rdst\times\Rdst$ let us denote $\varphi_{(x,\xi)}(t):=e^{2\pi i \xi t}\varphi(t-x)$.
If there exists a function $H \in L^1_w(\Rtdst)$ such that
\begin{align}
\label{eq_weyl_env}
\abs{\ip{T(\varphi_{(x,\xi)})}{\varphi_{(x',\xi')}}} \leq H(x'-x,\xi'-\xi), \qquad
(x,\xi),(x',\xi') \in \Rdst\times\Rdst,
\end{align}
then the exists $\sigma \in \Sjclass$ such that $T=\sigma^w$ on $\mathcal{S}(\Rdst)$.
\end{itemize}
\end{theorem}
As an application of Theorem \ref{th_weyl} we now prove Proposition \ref{prop_wiener}.
\begin{proof}[Proof of Proposition \ref{prop_wiener}]
With the notation of Theorem \ref{th_weyl} we use \eqref{eq_wiener_env} we to estimate
\begin{align*}
&\abs{\ip{T(\varphi_{(x,\xi)})}{\varphi_{(x',\xi')}}}
=\abs{V_\varphi T(\varphi_{(x,\xi)})(x',\xi')}
\\
&\qquad \leq
\int_{\Rtdst} \abs{V_\varphi \varphi_{(x,\xi)}(z'')} H((x',\xi')-z'') dz''
\\
&\qquad= \int_{\Rtdst} \abs{V_\varphi \varphi(z''-(x,\xi))} H((x',\xi')-z'') dz''
= (H*\abs{V_\varphi \varphi})((x',\xi')-(x,\xi)).
\end{align*}
Since $\varphi \in \Miw(\Rdst)$ and $H \in L^1_w(\Rtdst)$, we deduce that
$H*\abs{V_\varphi \varphi} \in L^1_w(\Rtdst)$.

Hence, Theorem \ref{th_weyl} implies that there exists $\sigma \in \Sjclass$ such that
$T \equiv \sigma^w$ on $\mathcal{S}(\Rdst)$. Since both operators are bounded on $L^2(\Rdst)$ it follows
that $T \equiv \sigma^w$ on $\LtRd$. By hypothesis $T=\sigma^w: \LtRd \to \LtRd$ is invertible. A new
application of Theorem \ref{th_weyl} implies that $\sigma^w: M^p_v(\Rdst) \to M^p_v(\Rdst)$ is invertible.
It is tempting to conclude that then $T:M^p_v(\Rdst) \to M^p_v(\Rdst)$ is invertible because it ``is''
$\sigma^w$. If $p<+\infty$ that conclusion is indeed correct because both operators coincide on the dense
space $\mathcal{S}$. The case $p=+\infty$ requires some carefulness. We now discuss this in detail.

We note that $M^p_v(\Rdst) \hookrightarrow M^\infty_{1/w}(\Rdst)$ and use the facts that $M^\infty_{1/w}(\Rdst)$ can
be identified with the dual space of the (separable) Banach space $M^1_w(\Rdst)$ and that
$\mathcal{S}$ is dense in $M^\infty_{1/w}(\Rdst)$ with respect to the weak* topology
(see \cite[Chapter 11]{gr01}). To conclude that $T=\sigma^w$ on $M^p_v(\Rdst)$ we show that both operators are
continuous with respect to the weak* topology of $M^\infty_{1/w}$.

Let $f \in M^p_v(\Rdst)$ and let us show that $T(f)=\sigma^w(f)$. Let $\sett{f_k: k \in \Nst} \subseteq
\mathcal{S}(\Rdst)$ be a sequence such that $f_k \longrightarrow f$ in the weak*-topology of 
$M^\infty_{1/w}(\Rdst)$. The operator $\sigma^w:M^\infty_{1/w}(\Rdst) \to M^\infty_{1/w}(\Rdst)$ is weak* continuous
because it is the adjoint of the operator $\overline{\sigma}^{w}: M^1_w(\Rdst) \to M^1_w(\Rdst)$. Hence
$T(f_k) = \sigma^w(f_k) \longrightarrow \sigma^w(f)$ in the weak*-topology of $M^\infty_{1/w}(\Rdst)$.
Let us note that this implies that
\begin{align}
\label{eq_xx}
V_\varphi T(f_k) (z) \longrightarrow V_\varphi \sigma^w(f)(z),
\mbox{ as } k \longrightarrow +\infty,
\mbox{ for all }z \in \Rtdst.
\end{align}
Indeed, if $z=(x,\xi) \in \Rdst \times \Rdst$ the function
$\varphi_{(x,\xi)} :=e^{2\pi i \xi \cdot}\varphi(\cdot-x)$ belongs to $M^1_w(\Rdst)$ and consequently
$V_\varphi T(f_k)(z) = \ip{T(f_k)}{\varphi_{(x,\xi)}} \longrightarrow 
\ip{\sigma^w(f)}{\varphi_{(x,\xi)}}= V_\varphi \sigma^w(f)(z)$.
Similarly, since $f_k \longrightarrow f$ in the weak* topology of $M^\infty_{1/w}$,
we know that $V_\varphi f_k (z) \longrightarrow V_\varphi f(z)$, for all $z \in \Rtdst$.

Using the enveloping condition in \eqref{eq_wiener_env}, we estimate for $z \in \Rtdst$
\begin{align*}
\abs{V_\varphi T(f)(z)- V_\varphi T(f_k)(z)}&
\leq \int_{\Rtdst} \abs{V_\varphi (f-f_k)(z')} H(z-z') dz'
\end{align*}
The integrand in the last expression tends to $0$ pointwise as $k \longrightarrow +\infty$.
In order to apply Lebesgue's dominated
convergence theorem
we show that the integrand is dominated by an integrable function. Since $H \in L^1_w(\Rtdst)$ it suffices to show that
$\sup_k \norm{V_\varphi (f-f_k)}_{L^\infty_{1/w}} < +\infty$. This is true because
$\norm{V_\varphi (f-f_k)}_{L^\infty_{1/w}}=\norm{f-f_k}_{M^\infty_{1/w}}$ and weak*-convergent sequences are bounded.
Hence, Lebesgue's dominated convergence theorem can be applied and we conclude that
$V_\varphi T(f_k) (z) \longrightarrow V_\varphi T(f)(z)$, for all $z \in \Rtdst$. Combining this with
\eqref{eq_xx} we conclude that $V_\varphi T(f) \equiv V_\varphi \sigma^w(f)$. Hence
$T(f)= \sigma^w(f)$, as desired.
\end{proof}

\end{document}